\documentclass[11pt]{article}
\usepackage{amsmath}
\usepackage{amsfonts,color}
\usepackage{mathrsfs}
\usepackage{amsfonts,openbib}
\usepackage{amsmath,amsthm,amssymb,
mathrsfs,amsfonts,cite}
\usepackage{times,lineno}
\usepackage{enumerate}
\usepackage{geometry}

\allowdisplaybreaks
\def\titlerunning#1{\gdef\titrun{#1}}
\makeatletter
\def\author#1{\gdef\autrun{\def\and{\unskip, }#1}\gdef\@author{#1}}

\makeatother

\def\keywords#1{\par\medskip
\noindent\textbf{Keywords.} #1}
\usepackage{makeidx}
\newtheorem{theorem}{Theorem}[section]
\newtheorem{corollary}[theorem]{Corollary}
\newtheorem{lemma}[theorem]{Lemma}

\newtheorem{proposition}[theorem]{Proposition}
\theoremstyle{definition}

\newtheorem{remark}[theorem]{Remark}

\newcommand{\R}{\mathbb R}
\newcommand{\Id}{\mathrm{Id}}
\newcommand{\dd}{\,\mathrm dx}

\newcommand{\Pcal}{\mathcal P}
\newcommand{\Ncal}{\mathcal N}
\newcommand{\Ecal}{\mathcal E}
\newcommand{\Fcal}{\mathcal F}
\newcommand{\Gform}{\mathfrak g}
\newcommand{\norm}[2]{\left\|#1\right\|_{#2}}
\newcommand{\ip}[2]{\left\langle #1,#2\right\rangle}
\newcommand{\HH}{\mathsf H}
\newcommand{\VV}{\mathsf V}
\newcommand{\XX}{\mathsf X}
\newcommand{\grad}{\nabla}
\newcommand{\supp}{\operatorname{supp}}

\newcommand{\dual}{\mathsf V'}
\numberwithin{equation}{section}
\allowdisplaybreaks

\makeatletter
\usepackage[usenames]{xcolor}   
\usepackage{hyperref} 			
\definecolor{ForestGreen}{rgb}{0.15,0.416,0.18}
\definecolor{EgyptBlue}{rgb}{0.063,0.2,0.65}
\hypersetup{
	colorlinks=true,
	linkcolor=EgyptBlue,
   citecolor=EgyptBlue,
   urlcolor=ForestGreen
}

\begin{document}
\pagestyle{plain}
\bibliographystyle{plain}
\titlerunning{AFA}
\begin{center}
{\Large\bfseries Global well-posedness for a damped two-dimensional fluid-beam system\par}
\vspace{1.2em}
{\large Weihua Wang\textsuperscript{1,*}\par}
\vspace{0.35em}
{\small School of Mathematical Science,\par}
{University of Sciences and Technology of China,\par}
{\small Hefei, Anhui 230000, P.R. China.\par}
{\small E-mail: \href{mailto:wwh133@ustc.edu.cn}{wwh133@ustc.edu.cn}\par}
\vspace{1.0em}
{\large Lu Yang\textsuperscript{2}\par}
\vspace{0.35em}
{\small Engineering University of the Chinese People's Armed Police Force,\par}
{\small Xi'an, Shaanxi 710086, P.R. China.\par}
{\small E-mail:\href{mailto:yanglu19970729@163.com}{yanglu19970729@163.com}\par}
\end{center}

\vspace{0.8em}
\begingroup
\renewcommand{\thefootnote}{\fnsymbol{footnote}}
\footnotetext[1]{Corresponding author:
E-mail: wwh133@ustc.edu.cn}
\endgroup
\begin{abstract}
We study the normalized surface system of an irrotational incompressible fluid of infinite depth beneath an elastic Euler--Bernoulli beam with structural damping. For every fixed damping coefficient $\gamma>0$, we establish global well-posedness for arbitrary initial data in $H^2(\R)\times\bigl(\dot H^{1/2}(\R)\cap\dot H^1(\R)\bigr)$. The main novelty is that the admissible data include finite energy initial potentials for which neither $\psi_0$ nor $u_0=\psi_0+G(\eta_0)\psi_0$ has an $L^2$ representative. We also prove the exact energy dissipation identity at the regularity of the constructed solutions. The main analytical difficulties are the possible failure of $L^2$ integrability at low frequencies and the justification of energy differentiation at the solution regularity. To overcome these difficulties,
we fix the initial low frequency background and establish local well-posedness by using homogeneous Dirichlet-Neumann estimates and a high frequency Fourier symmetrizer. An energy chain rule verifies the dissipation identity, together with coercivity, yields global continuation. Moreover, local uniqueness and stability estimates prove global uniqueness in the stated affine solution class and Lipschitz dependence on bounded sets over every finite time interval.
\end{abstract}
\section{Introduction}\label{sec:introduction}
\subsection{Description of the problem}
\noindent

We consider an ideal fluid of infinite depth whose free boundary is an elastic beam. Along the unbounded horizontal axis, the interface and fluid domain are
\begin{equation*}
  \Gamma(t)=\{(x,\eta(t,x)):x\in\R\},\qquad
  \Omega(t)=\{(x,y)\in\R^2:y<\eta(t,x)\}.
\end{equation*}
Incompressibility and irrotationality allow us to write the fluid velocity in terms of a harmonic potential:
\begin{equation}\label{1.1}
  \mathbf v=\nabla_{x,y}\Psi,\qquad
  \Delta_{x,y}\Psi=0\quad\text{in }\Omega(t).
\end{equation}
We define the surface potential by
\begin{equation}\label{1.2}
  \psi(t,x)=\Psi(t,x,\eta(t,x)).
\end{equation}
The bulk potential $\Psi$ is the variational harmonic extension of $\psi$ with finite Dirichlet energy, unique modulo additive constants. The velocity $\nabla_{x,y}\Psi$ is therefore uniquely determined.

The beam and fluid share the same normal velocity at the interface. Let
\[
  \mathbf n_\eta=
  \frac{(-\eta_x,1)}{\sqrt{1+\eta_x^2}}
\]
denote the outward unit normal to $\Omega(t)$. The kinematic boundary condition is then
\begin{equation}\label{1.3}
  \eta_t=\sqrt{1+\eta_x^2}\,
       \mathbf v|_{y=\eta}\cdot\mathbf n_\eta.
\end{equation}
The fluid pressure drives the beam at the interface. With zero exterior pressure, the beam's transverse displacement satisfies the linear Euler-Bernoulli equation with structural damping:
\begin{equation}\label{1.4}
  \rho_m\eta_{tt}-\gamma\partial_x^2\eta_t
       +\sigma\partial_x^4\eta=P|_{y=\eta}.
\end{equation}
Here $\rho_m>0$ denotes the beam mass per unit reference length, $\sigma>0$ its flexural rigidity, and $\gamma>0$ the structural damping coefficient.

The fluid satisfies the incompressible Euler equations under gravity:
\begin{equation}\label{1.5}
\left\{
\begin{aligned}
  &\mathbf v_t+(\mathbf v\cdot\nabla_{x,y})\mathbf v
       +\rho_f^{-1}\nabla_{x,y}P=-g_0\mathbf e_y
       &&\text{in }\Omega(t),\\
  &\operatorname{div}_{x,y}\mathbf v=0
       &&\text{in }\Omega(t),\\
  &\operatorname{curl}_{x,y}\mathbf v=0
       &&\text{in }\Omega(t),
\end{aligned}
\right.
\end{equation}
where $P$ is the fluid pressure, $\rho_f>0$ is the constant fluid density, $g_0>0$ is the gravitational acceleration, and $\mathbf e_y=(0,1)$.

We express the surface evolution through the Dirichlet-Neumann operator
\[
  G(\eta)\psi
  =(\Psi_y-\eta_x\Psi_x)|_{y=\eta},
\]
where $\Psi$ is the variational harmonic extension of $\psi$. Since this definition uses the unnormalized normal flux, boundary pairings use the horizontal measure $\mathrm dx$, and \eqref{1.3} becomes $\eta_t=G(\eta)\psi$.

We choose the temporal gauge of the velocity potential to remove the time-dependent constant in Bernoulli's equation. Taking its boundary trace and substituting the beam law \eqref{1.4}, we obtain
\begin{equation}\label{1.6}
\left\{
\begin{aligned}
  &\eta_t=G(\eta)\psi,\\
  &\psi_t+g_0\eta
  +\frac{\rho_m}{\rho_f}\eta_{tt}
  -\frac{\gamma}{\rho_f}\partial_x^2\eta_t
  +\frac{\sigma}{\rho_f}\partial_x^4\eta
  +\frac12\psi_x^2
  -\frac{(G(\eta)\psi+\eta_x\psi_x)^2}
         {2(1+\eta_x^2)}=0.
\end{aligned}
\right.
\end{equation}
We henceforth consider the normalized model with $\rho_f=\rho_m=\sigma=g_0=1$ and a fixed damping coefficient $\gamma>0$. Introducing
\begin{equation}\label{1.7}
  h=\eta_t,\qquad u=\psi+h,\qquad B=1+\partial_x^4,
\end{equation}
we obtain the surface system
\begin{equation}\label{1.8}
\begin{cases}
  \eta_t=h=G(\eta)\psi,\\
  u_t-\gamma\partial_x^2h+B\eta+\Ncal(\eta,\psi)=0,\\
  u=\psi+G(\eta)\psi,
\end{cases}
\end{equation}
where
\begin{equation}\label{1.9}
  \Ncal(\eta,\psi)
  =\frac12\psi_x^2
  -\frac{(G(\eta)\psi+\eta_x\psi_x)^2}
         {2(1+\eta_x^2)}.
\end{equation}
The operator $G(\eta)$ carries the nonlocal dependence on the fluid domain, while $u_t=\psi_t+\eta_{tt}$ incorporates beam inertia. The reduction for smooth solutions is verified in Section~\ref{sec:preliminaries}.

For viscous moving boundaries, Chambolle et al.~\cite{CDEG} studied a Navier-Stokes fluid coupled to an elastic plate, including the two dimensional fluid-membrane case. Grandmont and Hillairet~\cite{GH} obtained global strong solutions for a periodic viscous fluid beneath a viscoelastic beam and ruled out finite-time contact with the fixed bottom. Related work addresses deformable walls and shells~\cite{CCS,CS,LR2014,MC2013,MS2022}, large deformations~\cite{BKS2024}, viscous free surfaces with generalized surface energies~\cite{RT}, and perfectly elastic beams coupled to Navier-Stokes flow~\cite{SS2025}; broader accounts appear in \cite{Canic,Canic2020,GLN2014}. These viscous models benefit from regularity supplied by fluid viscosity, which is absent here.

For irrotational flow, harmonic extension reduces the bulk fluid problem to boundary data, while the Dirichlet-Neumann operator $G(\eta)$ carries the dependence on interface geometry. This approach originates in Zakharov's Hamiltonian formulation~\cite{Zakharov} and the surface formulation of Craig and Sulem~\cite{CraigSulem}. Water-wave analyses developed the operator's paralinearization and shape dependence~\cite{ABZ2011,ABZ2014,AM2009,CSS,Lannes2005,Lannes}. Elastic interfaces also occur in steady periodic waves beneath nonlinear membranes~\cite{BT2011} and in shell based models of nonlinear hydroelastic waves~\cite{PT}. For time-dependent hydroelastic interfaces, Ambrose and Siegel~\cite{AS2017} proved short time well-posedness in Sobolev spaces for a two dimensional potential flow model with a massless elastic membrane. Liu and Ambrose~\cite{LA2017} subsequently incorporated membrane inertia in a two fluid interfacial model and established well-posedness under a solvability condition for an associated integral equation. These results concern undamped elastic-sheet dynamics, whereas the present model consists of a single inviscid fluid of infinite depth beneath a structurally damped Euler-Bernoulli beam on $\R$.

Among inviscid results for undamped elastic free surfaces, Alazard, Kukavica, and Tuffaha~\cite{AKT} constructed global weak solutions for periodic, finite energy data; uniqueness in that class was not established in their work. On the whole line, Alazard, Shao, and Yang~\cite{ASY} established global well-posedness for $(\eta_0,u_0)\in H^2(\R)\times L^2(\R)$. They recast the undamped surface equations as a Schr\"odinger type evolution and combined dispersive estimates with low-regularity Dirichlet--Neumann estimates, obtaining continuous dependence and propagation of higher regularity.

The condition $u_0\in L^2(\R)$ in \cite[Theorem~1.1]{ASY} imposes an integrability requirement that is not implied by finite mechanical energy. Indeed, at a flat interface, the Fourier weight of $u$ in the kinetic energy vanishes at zero frequency. Finite energy initial data may therefore have no $L^2$ representative for either $\psi_0$ or $u_0$. For the damped system considered here, we establish global well-posedness in the natural energy space defined below, including these non-$L^2$ low frequency data. We also prove the exact energy dissipation identity at the regularity of the constructed solutions. Structural damping provides the high frequency estimate needed for local construction, while the energy identity and coercivity yield global continuation. The estimates depend on $\gamma$ and are not asserted to be uniform as $\gamma\downarrow0$.

\subsection{Main result}
\noindent

The mechanical energy determines the natural phase space. Let $\Lambda=|D|$. At a flat interface, $G(0)=\Lambda$ and $h=\Lambda\psi$, so the fluid and beam kinetic energies are $\frac12\|\psi\|_{\dot H^{1/2}}^2$ and $\frac12\|\psi\|_{\dot H^1}^2$, respectively. These terms lead to the homogeneous potential space
\begin{equation}\label{1.10}
  \Pcal=\dot H^{1/2}(\R)\cap\dot H^1(\R),\qquad
  \|\psi\|_{\Pcal}^2=
  \|\psi\|_{\dot H^{1/2}}^2+\|\psi_x\|_{L^2}^2.
\end{equation}
Here $\Pcal$ is the completion, under the norm in \eqref{1.10}, of
\[
\bigl\{\varphi\in\mathcal S(\R):
\operatorname{supp}\widehat\varphi
\subset\R\setminus\{0\}\bigr\}.
\]
Here and below, $\dot H^{1/2}(\R)$ denotes the completion of the test class displayed above under the norm
\[
\|\varphi\|_{\dot H^{1/2}}^2
=\int_\R |\xi|\,|\widehat\varphi(\xi)|^2\,\mathrm d\xi,
\]
and $\dot H^{-1/2}(\R)$ denotes its Hilbert dual. Elements of $\dot H^{1/2}(\R)$, and hence of $\Pcal$, are interpreted as tempered distributions modulo additive constants, with their action on zero-mean Schwartz test functions defined by Fourier duality. The homogeneous estimates proved below imply $h=G(\eta)\psi\in L^2$ for $\eta\in H^2$ and $\psi\in\Pcal$. If $u=\psi+h$ has an $L^2$ representative, then $\psi=u-h\in L^2$; together with $\psi_x\in L^2$, this gives $\psi\in H^1$. Consequently, $\psi=(\Id+G(\eta))^{-1}u$, consistently with the resolvent bound $L^2\to H^1$.

To accommodate homogeneous data, we separate the low frequency component of the initial potential from the evolution variable. Define the fixed background
\begin{equation}\label{1.11}
  f=\mathbb{P}\psi_0,
\end{equation}
where $\mathbb{P}$ projects onto $|\xi|\le1$, and write $u=f+v$ and $\psi=f+\zeta$. Since the high frequency part of $\psi_0$ lies in $H^1$ and the initial beam velocity lies in $L^2$, the initial remainder $v_0=u_0-f$ belongs to $L^2$. The correction $\zeta=\psi-f$ remains in $H^1$ in the solution class below, allowing an $L^2$ evolution without an additional integrability condition on $\psi_0$.

All function spaces below are over $\R$. The physical unknowns and initial data are real-valued; complexification is used only for Fourier and operator arguments. Set $\Ecal=H^2\times\Pcal$. The solution class is defined by
\begin{equation}\label{1.12}
\begin{aligned}
\eta&\in C([0,T];H^2)\cap L^2(0,T;H^3),\\
v:=u-f&\in C([0,T];L^2)\cap L^2(0,T;H^1),\\
\zeta:=\psi-f&\in C([0,T];H^1),
\end{aligned}
\end{equation}
where $f=\mathbb{P}\psi_0$ is fixed on $[0,T]$; at a restart, a new low-frequency background may be used. Since $u=\psi+h$, we have $h=v-\zeta$ and
\[
h\in C([0,T];L^2)\cap L^2(0,T;H^1).
\]
\noindent\textbf{Representatives and the solution concept.}
Fix a distributional representative of the time-independent background $f=\mathbb{P}\psi_0$. The high frequency part $(\Id-\mathbb{P})\psi_0$ has a unique $H^1$ representative, so this choice also fixes the representative of the initial potential. For a solution in \eqref{1.12}, define the time-dependent representatives by $\psi=f+\zeta$ and $u=f+v$. All equations in \eqref{1.8}, including the time derivative of $u$, are required to hold in $\mathcal D'((0,T)\times\R)$ for these representatives. The initial conditions are attained in $H^2\times\Pcal$, with $\zeta(0)=(\Id-\mathbb{P})\psi_0$. Replacing $f$ by $f+C$ for a constant $C$ shifts $\psi$ and $u$ by the same time-independent constant and leaves $v$, $\zeta$, and $h$ unchanged. This is the remaining gauge freedom. At a restart, the representative of the new background is chosen so that its difference from the old background is the $H^1$ function specified in Lemma~\ref{lem:affine-tools}.

Define the natural energy by
\begin{equation}\label{1.13}
E(\eta,\psi)=\frac12\left(
\|\eta\|_{L^2}^2+\|\eta_{xx}\|_{L^2}^2
+\|G(\eta)\psi\|_{L^2}^2+\Gform_\eta(\psi,\psi)\right).
\end{equation}
Here $\Gform_\eta(\psi,\psi)$ denotes the Dirichlet energy $\int_{\Omega_\eta}|\nabla_{x,y}\Psi|^2\,dx\,dy$ of the variational harmonic extension $\Psi$ of $\psi$. For homogeneous potentials, this definition does not imply that $\psi$ and $G(\eta)\psi$ admit an ordinary $L^2$ pairing.

\begin{theorem}\label{thm:main}
Fix $\gamma>0$. For every initial datum $(\eta_0,\psi_0)\in\Ecal$, the surface system \eqref{1.8} admits a global distributional solution in the sense just defined. It is unique within the affine solution class \eqref{1.12}, with potentials identified modulo time independent additive constants. The solution belongs to this class on every finite interval $[0,T]$. Writing $E(t)=E(\eta(t),\psi(t))$, it satisfies
\begin{equation}\label{1.14}
E(t)+\gamma\int_s^t\|h_x(\tau)\|_2^2\,\mathrm d\tau=E(s),
\qquad 0\le s\le t<\infty.
\end{equation}
If $E_0=E(\eta_0,\psi_0)$, then
\begin{equation}\label{1.15}
\sup_{t\ge0}\left(
\|\eta(t)\|_{H^2}
+\|\psi(t)\|_{\Pcal}
+\|h(t)\|_{L^2}
\right)\le C(E_0).
\end{equation}
On every bounded subset of $\Ecal$ and for every finite $T$, the solution map into $C([0,T];\Ecal)$ is Lipschitz, with a constant depending only on $\gamma$, the subset bound, and $T$. The initial data need not be small.
\end{theorem}
\begin{remark}
Let $\chi\in C_c^\infty(\R)$ be an even function supported near the origin and equal to one in a neighborhood of zero. For $1/2\le\alpha<1$, define $\psi_0$ by
\[
  \widehat{\psi_0}(\xi)=\chi(\xi)|\xi|^{-\alpha}.
\]
Then $\psi_0\in\Pcal$ because the defining homogeneous norms are finite near $\xi=0$ for $\alpha<1$, whereas $\psi_0\notin L^2$ because the $L^2$ integral diverges for $\alpha\ge1/2$. For the flat initial interface $\eta_0=0$, the initial energy is finite, yet
\[
  \widehat{u_0}(\xi)
  =(1+|\xi|)\widehat{\psi_0}(\xi),
\]
and hence $u_0=(1+\Lambda)\psi_0\notin L^2$. This example exhibits finite energy data outside the inhomogeneous $H^2(\R)\times L^2(\R)$ class considered in~\cite{ASY}.
\end{remark}

\subsection{ Strategy of the proof}
\noindent

We use the surface variable $u=\psi+\eta_t$ introduced in the undamped analysis of Alazard, Shao, and Yang~\cite{ASY}. At natural energy regularity, $u$ need not belong to $L^2$, and the quadratic nonlinearity $\Ncal$ is controlled in $L^1$ and hence in $H^{-1}$.

The fixed-graph elliptic theory starts from a variational harmonic extension for boundary data in $\dot H^{1/2}$. For an $H^1$ trace and $\eta\in H^2$, tangential regularity of the flattened equation yields an $L^2$ normal trace. A Rellich identity proves the equivalence of $\|G(\eta)\psi\|_{L^2}$ and $\|\psi_x\|_{L^2}$, while the resolvent $(\Id+G(\eta))^{-1}$ maps $L^2$ into $H^1$. We establish the shape derivative of $G(\eta)$ at the same graph regularity: although its boundary formula contains two terms initially defined only in $H^{-1}$, their sum is controlled in $L^2$ by cancellation. The fixed-graph estimates are proved in Appendix~\ref{app:elliptic}, and we extend them to homogeneous potentials in Section~\ref{sec:elliptic}.

For the evolution, we subtract $f=\mathbb{P}\psi_0$ and recover
\[
  \zeta=(\Id+G(\eta))^{-1}\bigl(v-G(\eta)f\bigr),
  \qquad h=v-\zeta.
\]
The homogeneous estimates give $G(\eta)f\in L^2$; since $v\in L^2$, the resolvent then gives $\zeta\in H^1$. The pair $Z=(B^{1/2}\eta,v)$ satisfies a semilinear system whose $H^{-1}$-valued forcing is locally Lipschitz in $Z$. A Fourier symmetrizer for the damped principal operator controls $Z$ in $C([0,T];L^2)\cap L^2(0,T;H^1)$ for forcing in $L^2(0,T;H^{-1})$. The resulting time-smallness factor yields a contraction without requiring $\psi\in L^2$.

Since the potential may have no $L^2$ representative, the energy identity requires justification at the constructed regularity. For each fixed background $f$, the energy is continuously Fr\'echet differentiable as a functional of $Z\in L^2(\R;\R^2)$, with $L^2$ gradient $Z+K_f(Z)$ and $K_f$ locally Lipschitz into $H^1(\R;\R^2)$. A chain rule in $H^1\subset L^2\subset H^{-1}$ then yields the energy dissipation identity.

Energy coercivity bounds the norms that determine the local lifespan uniformly along the solution. With a compatible background at each restart, this bound gives global continuation; iteration of local uniqueness and stability gives global uniqueness and Lipschitz dependence on every finite interval. These estimates require $\gamma>0$ and are not asserted to be uniform as $\gamma\downarrow0$. Moreover, since dissipation degenerates at low frequencies, the identity alone does not imply a quantitative decay rate.

We first derive the surface system and its flat linear scales in Section~\ref{sec:preliminaries}. Subsequently, the homogeneous estimates and affine evolution are established in Sections~\ref{sec:elliptic} and~\ref{sec:reduction}, respectively. Section~\ref{sec:local} establishes the local well-posedness. The energy identity and global continuation are proved in Section~\ref{sec:global}. Finally, Appendices~\ref{app:tools} and~\ref{app:elliptic} supply background compatibility and the fixed-graph elliptic theory.

\section{Preliminaries}\label{sec:preliminaries}
\noindent

We derive the boundary system and the energy law for smooth solutions, then identify the frequency scales used in the construction at natural energy regularity. All function spaces are over $\mathbb{R}$. We write $C_T X=C([0,T];X)$ and $L_T^p X=L^p(0,T;X)$, omitting the interval when it is clear. For Banach spaces $X$ and $Y$, we denote by $\mathcal L(X,Y)$ the space of bounded linear operators from $X$ to $Y$, equipped with the operator norm. Constants denoted by $C_M$ depend on the indicated bound $M$, while $C_{\gamma,M}$ also depends on the fixed damping coefficient $\gamma$.

\subsection{Derivation of the boundary system}
The surface formulation uses the unnormalized normal flux and pairs boundary quantities against $\mathrm{d}x$. The Dirichlet--Neumann operator is
\[
G(\eta)\psi
=\sqrt{1+\eta_x^2}\,\nabla\Psi\cdot \mathbf n_\eta
\quad\text{on } y=\eta,
\]
where
\[
\mathbf n_\eta=\frac{(-\eta_x,1)}{\sqrt{1+\eta_x^2}}
\]
is the unit normal to the graph. We now show how beam inertia and damping enter the equation for $u$.

\begin{proposition}\label{prop:boundary}
For a smooth irrotational incompressible fluid beneath a smooth graph, the kinematic condition, Bernoulli's equation in the zero temporal gauge, and the damped beam law \eqref{1.4} yield the surface system \eqref{1.8}. Conversely, a smooth solution of \eqref{1.8}, together with its harmonic extension, determines a pressure through Bernoulli's equation. The reconstructed velocity and pressure satisfy the Euler equations, while the surface system gives the kinematic condition and damped beam law.
\end{proposition}
\begin{proof}
Let $V=\Psi_x|_{y=\eta}$ and $\beta=\Psi_y|_{y=\eta}$. Applying the chain rule to $\psi(t,x)=\Psi(t,x,\eta(t,x))$ yields
\[
\psi_x=V+\eta_x\beta,\qquad
\psi_t=\Psi_t|_{y=\eta}+\beta\eta_t.
\]
The kinematic condition is given by
\[
\eta_t=\beta-\eta_x V=h.
\]
Thus, it holds that
\[
\beta=\frac{h+\eta_x\psi_x}{1+\eta_x^2},
\qquad
V=\psi_x-\eta_x\beta.
\]
With gravitational acceleration normalized to one, Bernoulli's equation takes the form
\[
\Psi_t+\tfrac12|\nabla\Psi|^2+y+P=0.
\]
Evaluating this equation at the interface and using $\psi_t=\Psi_t+\beta h$, we obtain
\[
\psi_t+\eta+P|_{y=\eta}
+\tfrac12(V^2+\beta^2)-\beta h=0.
\]
Substituting the expressions for $V$ and $\beta$ then yields
\[
\tfrac12(V^2+\beta^2)-\beta h
=\tfrac12\psi_x^2
-\frac{(h+\eta_x\psi_x)^2}{2(1+\eta_x^2)}
=\Ncal.
\]
Substituting
\[
P|_{y=\eta}
=\eta_{tt}-\gamma\eta_{txx}+\eta_{xxxx},
\]
and setting $u=\psi+\eta_t$, we obtain
\[
u_t-\gamma h_{xx}+B\eta+\Ncal=0.
\]
Together with the kinematic condition and the definition of $u$, this is \eqref{1.8}. Conversely, for a smooth surface solution, let $\Psi$ be its harmonic extension and define
\[
P=-\Psi_t-\tfrac12|\nabla\Psi|^2-y.
\]
Reversing the calculation recovers the damped beam law, while the first surface equation gives the kinematic condition. Taking the spatial gradient of Bernoulli's equation yields Euler's momentum equation; since $\Psi$ is harmonic, $\nabla\Psi$ is incompressible and irrotational. At energy regularity, we use the distributional surface formulation and the fixed-gauge representatives specified in Section~\ref{sec:introduction}. The reconstruction in this proposition establishes equivalence with the bulk Euler-beam system for smooth solutions; Theorem~\ref{thm:main} concerns the surface solution class \eqref{1.12}.
\end{proof}

\subsection{Mechanical energy and the dissipation law}
\noindent

For smooth solutions with sufficient spatial decay, the four terms in \eqref{1.13} represent gravitational potential, beam bending, beam kinetic, and fluid kinetic energy, respectively. By Green's identity,
\begin{equation}\label{2.1}
\Gform_\eta(\psi,\psi)
=\int_{\Omega_\eta}|\nabla\Psi|^2\,\mathrm{d}x\,\mathrm{d}y
=\int_{\mathbb{R}}\psi\,G(\eta)\psi\,\mathrm{d}x.
\end{equation}
The boundary integral on the right requires sufficiently regular, decaying traces. For homogeneous potentials as in \eqref{1.10}, we instead define the fluid energy by the variational Dirichlet form on the left.
\begin{proposition}
\label{prop:smooth-energy}
Every smooth solution of \eqref{1.8} with sufficient spatial decay satisfies
\begin{equation}\label{2.2}
\frac{\mathrm{d}}{\mathrm{d}t}E(\eta,\psi)
+\gamma\|\partial_x\eta_t\|_2^2=0.
\end{equation}
\end{proposition}

\begin{proof}
Write $G=G(\eta)$, and let $V$ and $\beta$ be the boundary velocities defined in Proposition~\ref{prop:boundary}. For an interface variation $k$, the shape derivative of the Dirichlet--Neumann operator is
\[
DG(\eta)[k]\psi=-G(k\beta)-\partial_x(kV).
\]
The extension of this formula to real-valued $\eta,k\in H^2$ and $\psi\in H^1$ is proved in Theorem~\ref{thm:B-shape}. It follows from self-adjointness and integration by parts that
\begin{align*}
\frac12\int_\R\psi DG(\eta)[k]\psi\dd
&=\frac12\int_\R k(\psi_xV-h\beta)\dd\\
&=\frac12\int_\R k\left(\psi_x^2-
\frac{(h+\eta_x\psi_x)^2}{1+\eta_x^2}\right)\dd
=\int_\R k\Ncal\dd.
\end{align*}
Consequently, differentiating the fluid energy gives
\[
\frac{\mathrm d}{\mathrm dt}\frac12\Gform_\eta(\psi,\psi)
=\int_\R h\psi_t\dd+\int_\R\Ncal\eta_t\dd.
\]
Adding the derivatives of the remaining energy terms and using $\eta_t=h$ and $u_t=\psi_t+h_t$, we obtain
\[
\frac{\mathrm{d}}{\mathrm{d}t}E
=\int_{\mathbb{R}} h\bigl(u_t+B\eta+\Ncal\bigr)\,\mathrm{d}x
=\gamma\int_{\mathbb{R}}h h_{xx}\,\mathrm{d}x
=-\gamma\|h_x\|_{L^2}^2,
\]
where the final equality follows by integration by parts.
\end{proof}

The fluid boundary work cancels with the corresponding change in beam energy.
The extension of the energy identity to the solution class~\eqref{1.12} is justified in Section~\ref{sec:energy} by using the differentiability of the energy functional and a chain rule. Once this identity is established, the nonnegative
energy terms imply
\[
\|\eta(t)\|_{H^2}\le 2\sqrt{E(0)},\qquad
\|h(t)\|_{L^2}\le \sqrt{2E(0)}.
\]
The homogeneous elliptic estimates in Section~\ref{sec:elliptic} also control $\|\psi(t)\|_{\Pcal}$; the full coercivity estimate is given in
Lemma~\ref{lem:coercivity}.

\subsection{The flat linearization and its frequency scales}
\noindent

At the flat interface, $G(0)=\Lambda=|D|$. Set
\[
a(D)=\Lambda(\Id+\Lambda)^{-1},\qquad
a(k)=\frac{k}{1+k}\quad(k=|\xi|).
\]
The linearization of \eqref{1.8} at $(\eta,\psi)=(0,0)$ is
\begin{equation}\label{2.3}
\begin{cases}
\eta_t=a(D)u,\\
u_t-\gamma\partial_x^2a(D)u+B\eta=0.
\end{cases}
\end{equation}
Eliminating $u$ at each nonzero Fourier mode gives
\begin{equation}\label{2.4}
\partial_t^2\widehat\eta+d_\gamma(k)\partial_t\widehat\eta
+\omega(k)^2\widehat\eta=0,
\qquad
d_\gamma(k)=\frac{\gamma k^3}{1+k},\quad
\omega(k)^2=\frac{k(1+k^4)}{1+k}.
\end{equation}
The corresponding characteristic roots are
\[
\lambda_\pm(k)=-\frac{d_\gamma(k)}2
\pm\frac12\sqrt{d_\gamma(k)^2-4\omega(k)^2}.
\]
At high frequencies, $\omega(k)\sim k^2$, which is the dispersion underlying the undamped Schr\"odinger-type formulation in~\cite{ASY}. For fixed $\gamma>0$, we also have $d_\gamma(k)\sim\gamma k^2$ and $\omega(k)^2\sim k^4$ as $k\to\infty$. Thus, as an unordered pair, the normalized characteristic roots $\lambda_\pm(k)/k^2$ converge to the roots of
\[
\mu^2+\gamma\mu+1=0.
\]
The quadratic energy of the flat linearized system is
\begin{align}\label{2.5}
E_{\mathrm{flat}}(\eta,\psi)
&=\frac12\int_{\R}\left[
(1+|\xi|^4)|\widehat\eta(\xi)|^2
+(|\xi|+|\xi|^2)|\widehat\psi(\xi)|^2
\right]\,\mathrm{d}\xi\notag\\
&=\frac12\int_{\R}\left[
(1+|\xi|^4)|\widehat\eta(\xi)|^2
+\frac{|\xi|}{1+|\xi|}
 |\widehat u(\xi)|^2
\right]\,\mathrm{d}\xi,
\end{align}
where $u=(\Id+\Lambda)\psi$. Since the weight $|\xi|/(1+|\xi|)$ vanishes at zero frequency, finite energy permits $u\notin L^2$. The nonlinear construction therefore fixes a low-frequency background. For the potential, we first define \eqref{2.5} on smooth Fourier data supported away from zero and then extend it to the weighted completion.

\subsection{Fourier cutoffs and elementary estimates}
\noindent

Throughout the paper, we use the unitary normalization of the Fourier transform:
\[
\widehat g(\xi)=(2\pi)^{-1/2}\int_\R e^{-ix\xi}g(x)\dd,
\qquad
\|g\|_{H^s}^2=\int_\R\langle\xi\rangle^{2s}
|\widehat g(\xi)|^2\,\mathrm d\xi,
\quad \langle\xi\rangle=(1+\xi^2)^{1/2}.
\]
The homogeneous potential space is defined by completion as in \eqref{1.10}. Let $b(D)=B^{1/2}$, where $b(\xi)=(1+\xi^4)^{1/2}$. The following estimates place the initial correction $\psi_0-\mathbb{P}\psi_0$ in $H^1$ and the quadratic boundary terms in $H^{-1}$.

\begin{lemma}[Fourier cutoffs and the homogeneous phase space]
\label{lem:cutoff-tools}
Let $\psi\in\Pcal$ and $f=\mathbb{P}\psi$. Then
\[
\|f\|_{\Pcal}\le\|\psi\|_{\Pcal},
\qquad
\|(\Id-\mathbb{P})\psi\|_{H^1}
\le\sqrt{2}\,\|\psi_x\|_{L^2}.
\]
Moreover, $H^1\hookrightarrow\Pcal$. Every $\psi\in\Pcal$ can be approximated in both defining homogeneous norms by smooth functions whose Fourier transforms are compactly supported away from zero. Suppose that $\supp\widehat f\subset[-1,1]$. For every $H>0$
and all nonnegative integers $j,k$ with $j+k\ge1$, we have
\[
\|\partial_x^j\partial_y^k e^{y\Lambda}f\|_{L^2}
+\|\partial_x^j\partial_y^k e^{y\Lambda}f\|_{L^\infty}
\le C_{j,k,H}\|f_x\|_{L^2},
\qquad |y|\le H.
\]
\end{lemma}
\begin{proof}
In Fourier variables, the squared $\Pcal$-norm is
\[
\|\psi\|_{\Pcal}^2
=\int_{\R} (|\xi|+\xi^2)
 |\widehat\psi(\xi)|^2\,\mathrm{d}\xi.
\]
Restricting the integral to $|\xi|\le1$ yields $\|\mathbb{P}\psi\|_{\Pcal}\le\|\psi\|_{\Pcal}$. On $|\xi|>1$, it holds that
\[
\|(\Id-\mathbb{P})\psi\|_{H^1}
\le\sqrt{2}\,\|\psi_x\|_{L^2},
\]
where we have used $1+\xi^2\le2\xi^2$. The inequality $|\xi|+\xi^2\le2(1+\xi^2)$ implies the continuous embedding $H^1\hookrightarrow\Pcal$.

Truncation of $\widehat\psi$ to $1/n\le|\xi|\le n$ converges in both components of the $\Pcal$-norm by the dominated convergence theorem. On each fixed annulus, the weight $|\xi|+\xi^2$ is bounded above and below by positive constants, so the truncated data can be approximated by smooth functions compactly supported away from zero. When $\psi$ is real, these approximations may be chosen conjugate symmetric, making their inverse Fourier transforms real. The resulting sequence converges to $\psi$ in the homogeneous completion $\Pcal$, whose elements are identified modulo constants as specified in Section~\ref{sec:introduction}.

Suppose that $\supp\widehat f\subset[-1,1]$, $j+k\ge1$, and $|y|\le H$. For $|\xi|\le1$,
\[
\bigl|(i\xi)^j|\xi|^k e^{y|\xi|}\bigr|
\le e^H|\xi|.
\]
Plancherel's theorem yields the $L^2$ estimate, while Fourier inversion and the Cauchy-Schwarz inequality on $[-1,1]$ give the $L^\infty$ estimate. Both of them depend only on $\|f_x\|_2$ and do not require $f\in L^2$.
\end{proof}

\begin{lemma}
\label{lem:product-tools}
In one spatial dimension, the following estimates hold:
\[
\|g\|_\infty\le C\|g\|_{H^1},\qquad
\|ab\|_{L^1}\le\|a\|_{L^2}\|b\|_{L^2},\qquad
\|a\|_{H^{-1}}\le C\|a\|_{L^1}.
\]
In particular, $H^2\hookrightarrow W^{1,\infty}$, and the product estimate gives
\[
\|g^2-\widetilde g^2\|_{L^1}
\le
\bigl(\|g\|_{L^2}+\|\widetilde g\|_{L^2}\bigr)
\|g-\widetilde g\|_{L^2}.
\]
Moreover, the Fourier multiplier $b(D)$ and its inverse satisfy
\[
b(D):H^1\to H^{-1},\qquad
b(D)^{-1}:H^s\to H^{s+2}
\quad(s\in\R).
\]
\end{lemma}
\begin{proof}
We first establish the one-dimensional Sobolev embedding. For $g\in\mathcal S(\R)$, using the Fourier inversion and the Cauchy-Schwarz inequality, we obtain
\[
\|g\|_{L^\infty}
\le (2\pi)^{-1/2}\|\widehat g\|_{L^1}
\le (2\pi)^{-1/2}
\left(\int_\R\langle\xi\rangle^{-2}\,\mathrm d\xi\right)^{1/2}
\|g\|_{H^1}
\le C\|g\|_{H^1}.
\]
If $g_n\to g$ in $H^1$, the same estimate makes $(g_n)$ uniformly Cauchy; its uniform limit agrees almost everywhere with $g$. Thus the estimate holds for every $g\in H^1$, with a bounded continuous representative. Applying it to both $g$ and $g_x$ when $g\in H^2$ yields
\[
\|g\|_{W^{1,\infty}}
\le C\bigl(\|g\|_{H^1}+\|g_x\|_{H^1}\bigr)
\le C\|g\|_{H^2}.
\]

If $a\in L^1$ and $\varphi\in H^1$, the embedding just proved implies
\[
\left|\int_\R a\overline{\varphi}\,\mathrm dx\right|
\le \|a\|_{L^1}\|\varphi\|_{L^\infty}
\le C\|a\|_{L^1}\|\varphi\|_{H^1}.
\]
Taking the supremum over $\|\varphi\|_{H^1}\le1$ shows that $a$ defines an element of $H^{-1}$ and that $\|a\|_{H^{-1}}\le C\|a\|_{L^1}$. Moreover, for
$g,\widetilde g\in L^2$, it holds that
\[
\begin{aligned}
\|g^2-\widetilde g^2\|_{L^1}
&=\|(g-\widetilde g)(g+\widetilde g)\|_{L^1}\\
&\le \|g-\widetilde g\|_{L^2}
     \|g+\widetilde g\|_{L^2}\\
&\le
\bigl(\|g\|_{L^2}+\|\widetilde g\|_{L^2}\bigr)
\|g-\widetilde g\|_{L^2}.
\end{aligned}
\]
Finally, $b(\xi)^2=1+\xi^4$ and
$\langle\xi\rangle^4=(1+\xi^2)^2$. The inequalities
\[
1+\xi^4\le(1+\xi^2)^2,
\qquad
2(1+\xi^4)-(1+\xi^2)^2=(1-\xi^2)^2\ge0
\]
therefore imply
\[
\tfrac12\langle\xi\rangle^4
\le b(\xi)^2
\le \langle\xi\rangle^4.
\]
Consequently, for $g\in H^1$,
\[
\|b(D)g\|_{H^{-1}}^2
=\int_\R \langle\xi\rangle^{-2}b(\xi)^2
       |\widehat g(\xi)|^2\,\mathrm d\xi
\le \|g\|_{H^1}^2.
\]
Since $b(\xi)>0$, its inverse multiplier is well defined; for every $s\in\R$ and $f\in H^s$,
\[
\|b(D)^{-1}f\|_{H^{s+2}}^2
=\int_\R
\frac{\langle\xi\rangle^{2s+4}}{b(\xi)^2}
|\widehat f(\xi)|^2\,\mathrm d\xi
\le 2\|f\|_{H^s}^2.
\]
These bounds, initially obtained for Schwartz functions, extend the multipliers continuously to the stated Sobolev spaces.
\end{proof}

\section{Elliptic estimates in the homogeneous potential space}\label{sec:elliptic}
\noindent

In Appendix~\ref{app:elliptic}, we establish the fixed-graph boundary trace and resolvent estimates used below, together with the cancellation that places the shape derivative in $L^2$. These results hold at the stated regularity without an external elliptic regularity theorem and are independent of the evolution and damping coefficient.

\begin{lemma}[Fixed-graph estimates]\label{lem:elliptic-input}
Let $\eta\in H^2(\R)$ be real valued. The variationally defined Dirichlet--Neumann operator $G(\eta)$ is a nonnegative self-adjoint operator on $L^2(\R)$, with operator domain $H^1(\R)$ and form domain $H^{1/2}(\R)$. If $\|\eta\|_{H^2}\le M$, the
following estimates hold uniformly, with constants depending only on $M$:
\begin{equation}\label{3.1}
C_M^{-1}\norm{\varphi_x}{L^2}
\le\norm{G(\eta)\varphi}{L^2}\le C_M\norm{\varphi_x}{L^2},
\quad\varphi\in H^1,
\end{equation}
and
\begin{equation}\label{3.2}
\norm{R_\eta g}{H^1}\le C_M\norm g{L^2},
\quad R_\eta=(\Id+G(\eta))^{-1}.
\end{equation}
For each real-valued $k\in H^2$ and each $\varphi\in H^1$, the shape derivative $DG(\eta)[k]\varphi$ exists in $L^2$ and satisfies
\begin{equation}\label{3.3}
\|DG(\eta)[k]\varphi\|_{L^2}
\le C_M\|k\|_{H^2}\|\varphi\|_{H^1}.
\end{equation}
It is given, in the sense of distributions, by
\begin{equation}\label{3.4}
DG(\eta)[k]\varphi
=-G(\eta)(k\beta)-\partial_x(k\nu),\qquad
\beta
=\frac{G(\eta)\varphi+\eta_x\varphi_x}{1+\eta_x^2},
\qquad
\nu=\varphi_x-\eta_x\beta,
\end{equation}
where $G(\eta)$ in the first term denotes its bounded extension from $L^2$ to $H^{-1}$ by duality. Although the two terms on the right are individually defined in $H^{-1}$, their sum belongs to $L^2$. For real-valued graphs $\eta_1,\eta_2$ satisfying
$\|\eta_j\|_{H^2}\le M$, $j=1,2$, we have
\begin{align}
\|(G(\eta_1)-G(\eta_2))\varphi\|_{L^2}
&\le C_M\|\eta_1-\eta_2\|_{H^2}\|\varphi\|_{H^1},
\label{3.5}\\
\|(R_{\eta_1}-R_{\eta_2})g\|_{H^1}
&\le C_M\|\eta_1-\eta_2\|_{H^2}\|g\|_{L^2},
\label{3.6}
\end{align}
for all $\varphi\in H^1$ and $g\in L^2$, where
$R_\eta=(\Id+G(\eta))^{-1}$.
\end{lemma}
\begin{proof}
The assertions follow from the fixed-graph theory developed in Appendix~\ref{app:elliptic}. The variational harmonic extension and its tangential regularity, established in Lemma~\ref{lem:B-shift}, provide the basis for identifying
$G(\eta)\varphi$ with the $L^2$ conormal trace for $\varphi\in H^1$ in Proposition~\ref{prop:B-DN}. Together with the Rellich estimate in Proposition~\ref{prop:B-Rellich}, this yields the two-sided comparison between $\|G(\eta)\varphi\|_2$ and $\|\varphi_x\|_2$. By Proposition~\ref{prop:B-resolvent},
$G(\eta)$ is nonnegative and self-adjoint on $L^2$, with operator domain $H^1$ and form domain $H^{1/2}$, and $\Id+G(\eta):H^1\to L^2$ is an isomorphism whose inverse
$R_\eta$ is bounded from $L^2$ to $H^1$. The dependence on the graph is quantified by
\eqref{B14} and \eqref{B18}, which give the stated Lipschitz estimates, and by Theorem~\ref{thm:B-shape}, which establishes continuous Fr\'echet differentiability
of $\eta\mapsto G(\eta)$ in $\mathcal L(H^1,L^2)$. The latter theorem also gives a distributional formula for the derivative and an $L^2$ remainder estimate that is quadratic in the $H^2$ norm of the graph increment.
\end{proof}

The preceding $L^2$ estimates for the Dirichlet-Neumann operator were established for boundary traces in $H^1$. We now extend $G(\eta)$ to the homogeneous potential space $\Pcal$, whose elements need not belong to $L^2$, and establish the corresponding graph-continuity and Dirichlet-form estimates.
\begin{lemma}[Homogeneous estimates for the Dirichlet-Neumann operator]\label{lem:homogeneous}
For each real-valued $\eta\in H^2$, the Dirichlet-Neumann operator extends uniquely to a bounded linear map $G(\eta):\Pcal\to L^2$. If $\|\eta\|_{H^2}\le M$, then
\begin{equation}\label{3.7}
C_M^{-1}\|\psi_x\|_{L^2}
\le \|G(\eta)\psi\|_{L^2}
\le C_M\|\psi_x\|_{L^2},
\end{equation}
for every $\psi\in\Pcal$. Moreover, for real valued graphs $\eta_1,\eta_2$ satisfying
$\|\eta_j\|_{H^2}\le M$, $j=1,2$, we have
\begin{equation}\label{3.8}
\|(G(\eta_1)-G(\eta_2))\psi\|_{L^2}
\le C_M\|\eta_1-\eta_2\|_{H^2}\|\psi_x\|_{L^2},
\qquad \psi\in\Pcal.
\end{equation}
Thus these bounds only depend on $\|\psi_x\|_2$; the potential itself need not lie in $L^2$.
\end{lemma}
\begin{proof}
Choose $\psi_n$ with $\widehat{\psi_n}\in C_c^\infty(\R\setminus\{0\})$ and $\psi_n\to\psi$ in both $\dot H^{1/2}$ and $\dot H^1$. The upper bound in \eqref{3.1}, applied to $\psi_n-\psi_m$, shows that $G(\eta)\psi_n$ is Cauchy in $L^2$; denote its limit by $q$. Passing both bounds in \eqref{3.1} to the limit gives \eqref{3.7} with $q$ in place of $G(\eta)\psi$. For $\varphi\in C_c^\infty(\R)$, it follows from the variational definition that
\[
\int_{\R}G(\eta)\psi_n\,\varphi\,\mathrm{d}x
=\Gform_\eta(\psi_n,\varphi).
\]
The independent variational argument in Lemma~\ref{lem:form} below shows that convergence in $\dot H^{1/2}$ implies convergence of the corresponding harmonic extensions in Dirichlet energy. Passing to the limit therefore yields
\[
\int_{\R}q\,\varphi\,\mathrm{d}x
=\Gform_\eta(\psi,\varphi).
\]
Thus $q$ is the variational Dirichlet-Neumann operator applied to $\psi$ and is independent of the approximating sequence.

To prove \eqref{3.8}, split $\psi=f+g$ with $f=\mathbb{P}\psi$ and $g=(\Id-\mathbb{P})\psi$. Since $\|g\|_{H^1}\le C\|\psi_x\|_{L^2}$, the high frequency part $g$ can be bounded by estimate \eqref{3.5}. For the low frequency part, assume temporarily that $\widehat f$ is also supported away from zero, and define
\[
\begin{aligned}
F(x,y)&=e^{y\Lambda}f(x),\qquad
 d_\eta(x)=F(x,\eta(x))-f(x),\\
W_\eta(x)&=F_y(x,\eta(x))-\eta_x(x)F_x(x,\eta(x)).
\end{aligned}
\]
For all nonnegative integers $j,k$ with $j+k\ge1$, we have
\begin{equation}\label{3.9}
\|\partial_x^j\partial_y^kF(\cdot,y)\|_{L^2}
+\|\partial_x^j\partial_y^kF(\cdot,y)\|_{L^\infty}
\le C_{j,k,M}\|f_x\|_{L^2},
\qquad |y|\le CM.
\end{equation}
Plancherel's theorem and the multiplier bound give the $L^2$ estimate in \eqref{3.9}, while Bernstein's inequality gives the $L^\infty$ estimate because the $x$-Fourier support remains in $[-1,1]$.

The identity
\[
d_\eta(x)=\int_0^{\eta(x)}F_y(x,y)\,\mathrm{d}y
\]
implies
\[
\partial_x d_\eta(x)
=\int_0^{\eta(x)}F_{xy}(x,y)\,\mathrm{d}y
+\eta_x(x)F_y(x,\eta(x)).
\]
Equation~\eqref{3.9} applied to this identity gives $\|d_\eta\|_{H^1}\le C_M\|f_x\|_2$. To estimate derivatives of $F$ on the graph, write
\[
F_y(x,\eta(x))
=F_y(x,0)+\int_0^{\eta(x)}F_{yy}(x,y)\,\mathrm{d}y.
\]
The integral term is bounded in $L^2$ by
\[
\|\eta\|_{L^2}
\sup_{|y|\le CM}\|F_{yy}(\cdot,y)\|_{L^\infty}.
\]
Together with \eqref{3.9}, this yields $\|F_y(\cdot,\eta(\cdot))\|_{L^2}\le C_M\|f_x\|_{L^2}$; the same argument applies to $F_x$. Since $\eta_x\in L^\infty$, it follows from the definition of $W_\eta$ that
\begin{equation}\label{3.10}
\|d_\eta\|_{H^1}+\|W_\eta\|_{L^2}
\le C_M\|f_x\|_{L^2}.
\end{equation}
For two graphs $\eta_1,\eta_2$, by the fundamental theorem of calculus, we have
\[
d_{\eta_1}(x)-d_{\eta_2}(x)
=(\eta_1(x)-\eta_2(x))
\int_0^1
F_y\bigl(x,\eta_2(x)+\theta(\eta_1(x)-\eta_2(x))\bigr)
\,\mathrm{d}\theta.
\]
Set $\delta\eta=\eta_1-\eta_2$ and $\eta_\theta=\eta_2+\theta\delta\eta$. Differentiation in $x$ gives
\[
\partial_x(d_{\eta_1}-d_{\eta_2})
=\delta\eta_x\int_0^1 F_y(x,\eta_\theta)\,\mathrm{d}\theta
+\delta\eta\int_0^1
\bigl(F_{xy}(x,\eta_\theta)
+\eta_{\theta x}F_{yy}(x,\eta_\theta)\bigr)
\,\mathrm{d}\theta.
\]
Equation~\eqref{3.9} bounds each displayed derivative of $F$ in $L^\infty$ by $C_M\|f_x\|_{L^2}$, uniformly for $\theta\in[0,1]$, while $\|\eta_{\theta x}\|_\infty\le C_M$. Together with the formula for $d_{\eta_1}-d_{\eta_2}$, these bounds yield the required $H^1$ estimate. For $W_\eta$, we have
\begin{align*}
W_{\eta_1}-W_{\eta_2}
={}&\delta\eta\int_0^1
F_{yy}(x,\eta_\theta)\,\mathrm{d}\theta
-\delta\eta_x F_x(x,\eta_1)\\
&-\eta_{2x}\delta\eta\int_0^1
F_{xy}(x,\eta_\theta)\,\mathrm{d}\theta.
\end{align*}
Each product is bounded in $L^2$, which implies
\begin{equation}\label{3.11}
\|d_{\eta_1}-d_{\eta_2}\|_{H^1}
+\|W_{\eta_1}-W_{\eta_2}\|_{L^2}
\le C_M\|\eta_1-\eta_2\|_{H^2}\|f_x\|_{L^2}.
\end{equation}
Choose $H\ge\max\{0,\sup\eta\}$. Since $\supp\widehat f\subset[-1,1]$, it follows from Plancherel's identity that
\[
\int_{\R\times(-\infty,H)}|\nabla F|^2
\,\mathrm{d}x\,\mathrm{d}y
=\int_{\R}|\xi|e^{2H|\xi|}
|\widehat f(\xi)|^2\,\mathrm{d}\xi
\le e^{2H}\|f\|_{\dot H^{1/2}}^2.
\]
Thus $F$ is a finite energy harmonic function in $\Omega_\eta$ with boundary trace $f+d_\eta$. By uniqueness of the variational harmonic extension, its unnormalized normal derivative satisfies
\[
G(\eta)(f+d_\eta)=W_\eta.
\]
Consequently,
\begin{equation}\label{3.12}
G(\eta)f=W_\eta-G(\eta)d_\eta.
\end{equation}
Subtracting \eqref{3.12} for the two graphs, we have
\[
(G(\eta_1)-G(\eta_2))f
=W_{\eta_1}-W_{\eta_2}
-\bigl(G(\eta_1)d_{\eta_1}-G(\eta_2)d_{\eta_2}\bigr).
\]
Decompose the parenthesized difference as
\[
G(\eta_1)(d_{\eta_1}-d_{\eta_2})
+(G(\eta_1)-G(\eta_2))d_{\eta_2}.
\]
Applying \eqref{3.1}, \eqref{3.5}, \eqref{3.10}, and \eqref{3.11} to these terms yields
\[
\|G(\eta_1)d_{\eta_1}-G(\eta_2)d_{\eta_2}\|_{L^2}
\le C_M\|\eta_1-\eta_2\|_{H^2}\|f_x\|_{L^2}.
\]
This bound and \eqref{3.11} control the graph difference for low frequency data truncated away from zero. Convergence in $\dot H^1$, together with \eqref{3.7}, removes the extra cutoff. Combining this result with the high frequency estimate proves \eqref{3.8}, without requiring an $L^2$ bound on $f$.
\end{proof}

\begin{lemma}[Equivalence and stability of the Dirichlet form]\label{lem:form}
For every real valued $\eta\in H^2$ satisfying $\|\eta\|_{H^2}\le M$, the Dirichlet form satisfies
\begin{equation}\label{3.13}
C_M^{-1}\norm\psi{\dot H^{1/2}}^2
\le\Gform_\eta(\psi,\psi)
\le C_M\norm\psi{\dot H^{1/2}}^2.
\end{equation}
The map $(\eta,\psi)\mapsto\Gform_\eta(\psi,\psi)$ is continuous on $H^2\times\dot H^{1/2}$ for real valued $\eta$. More precisely, if $\|\eta_j\|_{H^2}\le M$ and $\psi_j\in\dot H^{1/2}$ for $j=1,2$, then
\begin{align}\label{3.14}
|\Gform_{\eta_1}(\psi_1,\psi_1)
       -\Gform_{\eta_2}(\psi_2,\psi_2)|\notag
&\le C_M(\norm{\psi_1}{\dot H^{1/2}}
                 +\norm{\psi_2}{\dot H^{1/2}})
                    \norm{\psi_1-\psi_2}{\dot H^{1/2}}\notag\\
&+C_M\norm{\eta_1-\eta_2}{H^2}
                         \norm{\psi_2}{\dot H^{1/2}}^2.
\end{align}
\end{lemma}
\begin{proof}
Let $\Pi=\R\times(-\infty,0)$ and flatten the fluid domain by $(x,z)\mapsto(x,z+\eta(x))$. The transformed Dirichlet energy has coefficient matrix
\[
A_\eta(x)
=\begin{pmatrix}
1&-\eta_x\\
-\eta_x&1+\eta_x^2
\end{pmatrix},
\qquad
\det A_\eta=1.
\]
The matrix is symmetric and positive definite, with trace $2+\eta_x^2$ and determinant $1$. Consequently,
\[
(2+\|\eta_x\|_\infty^2)^{-1}|\xi|^2
\le \xi^{\mathsf T}A_\eta(x)\xi
\le (2+\|\eta_x\|_\infty^2)|\xi|^2,
\qquad \xi\in\R^2.
\]
For $\psi\in\dot H^{1/2}$, the flat Poisson extension $P\psi=e^{z\Lambda}\psi$ satisfies
\[
\|\nabla P\psi\|_{L^2(\Pi)}^2
=\|\psi\|_{\dot H^{1/2}}^2.
\]
Let $\mathcal V_0$ denote the closed subspace of zero-trace elements in the homogeneous energy space on $\Pi$, understood modulo constants. The Lax-Milgram theorem then gives a unique $a\in\mathcal V_0$ satisfying
\[
\int_\Pi A_\eta\nabla a\cdot\nabla\theta
\,\mathrm{d}x\,\mathrm{d}z
=-\int_\Pi A_\eta\nabla P\psi\cdot\nabla\theta
\,\mathrm{d}x\,\mathrm{d}z,
\qquad \theta\in\mathcal V_0.
\]
The function $w_\eta=P\psi+a$ minimizes the transformed Dirichlet energy among extensions of $\psi$. In physical coordinates,
\[
\Psi(x,y)=w_\eta(x,y-\eta(x))
\]
is the variational harmonic extension of $\psi$ in $\Omega_\eta$, understood modulo constants. This construction defines the bilinear Dirichlet form for boundary traces in $\dot H^{1/2}$ without requiring them to belong to $L^2$. It follows from the flat variational principle that
\[
\|\psi\|_{\dot H^{1/2}}^2
\le \|\nabla w\|_{L^2(\Pi)}^2,
\]
for every finite-energy extension $w$ of $\psi$. Applying this inequality to $w_\eta$ and using ellipticity gives the lower bound in \eqref{3.13}; taking $P\psi$ as a competitor gives the upper bound. To compare graphs with the same trace $\psi$, write $w_j=w_{\eta_j}[\psi]$ and set
\[
Q_j=\int_\Pi A_{\eta_j}\nabla w_j\cdot\nabla w_j
\,\mathrm{d}x\,\mathrm{d}z,
\qquad j=1,2.
\]
Using $w_2$ as a competitor for the first minimization problem, we obtain
\[
Q_1-Q_2
\le\int_\Pi
(A_{\eta_1}-A_{\eta_2})\nabla w_2\cdot\nabla w_2
\,\mathrm{d}x\,\mathrm{d}z.
\]
Interchanging the indices yields the reverse bound. The coefficient estimate
\[
\|A_{\eta_1}-A_{\eta_2}\|_\infty
\le C_M\|\eta_1-\eta_2\|_{H^2}
\]
combines with ellipticity and \eqref{3.13} to obtain
\[
|Q_1-Q_2|
\le C_M\|\eta_1-\eta_2\|_{H^2}
\|\psi\|_{\dot H^{1/2}}^2.
\]
For two traces on the same graph, it follows from symmetry and bilinearity that
\[
\Gform_\eta(\psi_1,\psi_1)-\Gform_\eta(\psi_2,\psi_2)
=\Gform_\eta(\psi_1-\psi_2,\psi_1+\psi_2).
\]
Cauchy-Schwarz for the Dirichlet form and the upper bound in \eqref{3.13} imply
\[
\bigl|\Gform_\eta(\psi_1,\psi_1)
       -\Gform_\eta(\psi_2,\psi_2)\bigr|\le
C_M\bigl(\|\psi_1\|_{\dot H^{1/2}}
       +\|\psi_2\|_{\dot H^{1/2}}\bigr)
       \|\psi_1-\psi_2\|_{\dot H^{1/2}}.
\]
Comparing the traces at $\eta_1$ and then the graphs with trace $\psi_2$ yields \eqref{3.14}. By linearity, $w_\eta[\psi_1]-w_\eta[\psi_2]$ is the extension of $\psi_1-\psi_2$. Uniform ellipticity of $A_\eta$ and \eqref{3.13} therefore give
\[
\|\nabla(w_\eta[\psi_1]-w_\eta[\psi_2])\|_{L^2(\Pi)}
\le C_M\|\psi_1-\psi_2\|_{\dot H^{1/2}}.
\]
This establishes the convergence of harmonic extensions used in Lemma~\ref{lem:homogeneous} directly from the variational construction, independently of the Dirichlet-Neumann trace estimates.
\end{proof}
\section{Affine reduction with a fixed low-frequency background}\label{sec:reduction}
Fix a real valued $f\in\Pcal$ with $\supp\widehat f\subset[-1,1]$. For a real-valued
$\eta\in H^2$ and $v\in L^2$, set
\begin{equation}\label{4.1}
u=f+v,\qquad
\psi=f+\zeta,\qquad
\zeta=R_\eta\bigl(v-G(\eta)f\bigr).
\end{equation}
By \eqref{3.7}, $G(\eta)f\in L^2$, and the resolvent estimate places $\zeta$ in $H^1$ even if $f\notin L^2$. Consequently, $(\Id+G(\eta))\zeta=v-G(\eta)f$ can be written as
\begin{equation}\label{4.2}
G(\eta)(f+\zeta)=v-\zeta=:h.
\end{equation}

\begin{lemma}\label{lem:chart}
For $j=1,2$, define $\zeta_j$, $\psi_j$, and $h_j$ from $(\eta_j,v_j,f_j)$ by \eqref{4.1}--\eqref{4.2}. Suppose
\[
\|\eta_j\|_{H^2}+\|v_j\|_{L^2}+\|\partial_x f_j\|_{L^2}\le R,
\qquad j=1,2.
\]
Then
\begin{equation}\label{4.3}
\|\zeta_j\|_{H^1}
\le C_R\bigl(\|v_j\|_{L^2}+\|\partial_x f_j\|_{L^2}\bigr),
\qquad j=1,2,
\end{equation}
and
\begin{equation}\label{4.4}
\|\zeta_1-\zeta_2\|_{H^1}
\le C_R\bigl(
\|\eta_1-\eta_2\|_{H^2}
+\|v_1-v_2\|_{L^2}
+\|\partial_x(f_1-f_2)\|_{L^2}
\bigr).
\end{equation}
The nonlinearity in \eqref{1.9} also satisfies
\begin{equation}\label{4.5}
\|\Ncal(\eta_j,\psi_j)\|_{L^1}\le C_R,
\qquad j=1,2,
\end{equation}
and
\begin{equation}\label{4.6}
\|\Ncal(\eta_1,\psi_1)-\Ncal(\eta_2,\psi_2)\|_{L^1}
\le C_R\bigl(
\|\eta_1-\eta_2\|_{H^2}
+\|v_1-v_2\|_{L^2}
+\|\partial_x(f_1-f_2)\|_{L^2}
\bigr).
\end{equation}
\end{lemma}
\begin{proof}
Write $G_j=G(\eta_j)$ and $R_j=R_{\eta_j}$. It follows from \eqref{3.2} and \eqref{3.7} that
\[
\|\zeta_j\|_{H^1}
\le C_R\bigl(\|v_j\|_{L^2}+\|G_jf_j\|_{L^2}\bigr)
\le C_R\bigl(\|v_j\|_{L^2}+\|\partial_x f_j\|_{L^2}\bigr),
\]
Thus \eqref{4.3} follows. Subtracting the reconstruction formulas, it holds:
\begin{align*}
\zeta_1-\zeta_2
=(R_1-R_2)(v_2-G_2f_2)
+R_1\bigl(
v_1-v_2-G_1(f_1-f_2)-(G_1-G_2)f_2
\bigr).
\end{align*}
Using $\|v_2-G_2f_2\|_{L^2}\le C_R$, the resolvent estimates \eqref{3.6} and \eqref{3.2}, together with Lemma~\ref{lem:homogeneous}, applied to the decomposition
above yield \eqref{4.4}. Let
\[
\delta
=\|\eta_1-\eta_2\|_{H^2}
+\|v_1-v_2\|_{L^2}
+\|\partial_x(f_1-f_2)\|_{L^2},
\]
and introduce
\[
a_j=\partial_x\psi_j,\qquad
s_j=\partial_x\eta_j,\qquad
c_j=(1+s_j^2)^{-1},\qquad
t_j=h_j+s_ja_j.
\]
The preceding estimates, together with $H^2\hookrightarrow W^{1,\infty}$, imply
\[
\|a_j\|_{L^2}+\|h_j\|_{L^2}+\|t_j\|_2+\|s_j\|_{L^\infty}
\le C_R
\]
and
\[
\|a_1-a_2\|_{L^2}+\|h_1-h_2\|_{L^2}
+\|s_1-s_2\|_{L^\infty}
\le C_R\delta.
\]
Using
\[
t_1-t_2
=(h_1-h_2)+s_1(a_1-a_2)+(s_1-s_2)a_2,
\]
we obtain $\|t_1-t_2\|_{L^2}\le C_R\delta$. Since $s\mapsto(1+s^2)^{-1}$ has bounded derivative, we also have $\|c_1-c_2\|_{L^\infty}\le C_R\delta$.
Let $\Ncal_j=\Ncal(\eta_j,\psi_j)$. Since $2\Ncal_j=a_j^2-c_jt_j^2$ and $0<c_j\le1$,
\[
\|\Ncal_j\|_{L^1}
\le \tfrac12\bigl(\|a_j\|_{L^2}^2+\|t_j\|_{L^2}^2\bigr)
\le C_R,
\]
which proves \eqref{4.5}. To estimate the difference, write
\[
\begin{aligned}
2(\Ncal_1-\Ncal_2)
={}&(a_1-a_2)(a_1+a_2)\\
&-c_1(t_1-t_2)(t_1+t_2)
-(c_1-c_2)t_2^2.
\end{aligned}
\]
The $L^2\cdot L^2\to L^1$ product estimate then yields
\[
\begin{aligned}
2\|\Ncal_1-\Ncal_2\|_{L^1}
&\le{}
\|a_1-a_2\|_{L^2}\bigl(\|a_1\|_{L^2}+\|a_2\|_{L^2}\bigr)\\
&+\|t_1-t_2\|_{L^2}\bigl(\|t_1\|_{L^2}+\|t_2\|_{L^2}\bigr)\\
&+\|c_1-c_2\|_{L^\infty}\|t_2\|_{L^2}^2\\
&\le C_R\delta.
\end{aligned}
\]
This proves \eqref{4.6} without an $L^\infty$ bound on $a_j$ or $h_j$.
\end{proof}

Since $f$ is time independent, the surface system takes the form
\begin{equation}\label{4.7}
\begin{cases}
\eta_t=v-\zeta,\\
v_t-\gamma v_{xx}+B\eta
=-\gamma\zeta_{xx}-\Ncal(\eta,f+\zeta).
\end{cases}
\end{equation}
With $b(D)=B^{1/2}$ and $Z=(b(D)\eta,v)$, this becomes
\begin{equation}\label{4.8}
Z_t=L_\gamma(D)Z+\Fcal_f(Z),
\qquad
L_\gamma(\xi)=
\begin{pmatrix}
0&b(\xi)\\
-b(\xi)&-\gamma\xi^2
\end{pmatrix},
\end{equation}
where
\begin{equation}\label{4.9}
\Fcal_f(Z)=
\begin{pmatrix}
-b(D)\zeta\\
-\gamma\partial_x^2\zeta-\Ncal(\eta,f+\zeta)
\end{pmatrix}.
\end{equation}
Here $\eta=b(D)^{-1}Z_1$, $v=Z_2$, and \eqref{4.1} reconstructs $\zeta$. For evolution estimates, we take $L_\gamma(D)$ as the principal part and retain the resolvent terms, including their linear contributions, in $\Fcal_f$. Thus $L_\gamma(D)$ is not the full linearization of \eqref{1.8}.
\begin{corollary}\label{cor:forcing}
On the bounded sets in Lemma~\ref{lem:chart},
\begin{align}
\norm{\Fcal_f(Z)}{H^{-1}}&\le C_{\gamma,R},\label{4.10}\\
\norm{\Fcal_{f_1}(Z_1)-\Fcal_{f_2}(Z_2)}{H^{-1}}
&\le C_{\gamma,R}\bigl(\norm{Z_1-Z_2}{L^2}
+\norm{\partial_x(f_1-f_2)}{L^2}\bigr).\label{4.11}
\end{align}
\end{corollary}
\begin{proof}
Lemma~\ref{lem:product-tools} gives the bounded maps $b(D):H^1\to H^{-1}$ and $L^1\hookrightarrow H^{-1}$. Moreover, integration by parts yields
\[
\|\partial_x^2g\|_{H^{-1}}
\le\|g_x\|_{L^2}
\le\|g\|_{H^1}.
\]
Therefore
\[
\|\Fcal_f(Z)\|_{H^{-1}}
\le C(1+\gamma)\|\zeta\|_{H^1}+C\|\Ncal\|_{L^1}.
\]
Combining this with \eqref{4.3} and \eqref{4.5} proves \eqref{4.10}. Applying these bounds to the difference of the nonlinear forcing terms yields
\[
\begin{aligned}
\|\Fcal_{f_1}(Z_1)-\Fcal_{f_2}(Z_2)\|_{H^{-1}}
&\le C(1+\gamma)\|\zeta_1-\zeta_2\|_{H^1}\\
&+C\|\Ncal(\eta_1,\psi_1)
-\Ncal(\eta_2,\psi_2)\|_{L^1}.
\end{aligned}
\]
Finally, combine \eqref{4.4} and \eqref{4.6} with
\[
\|\eta_1-\eta_2\|_{H^2}
\le C\|Z_1-Z_2\|_{L^2},
\]
to obtain \eqref{4.11}.
\end{proof}

\section{A parabolic estimate and local well-posedness}\label{sec:local}
For $T>0$, define
\[
\XX_T
=C([0,T];L^2(\R;\R^2))
\cap L^2(0,T;H^1(\R;\R^2)),
\]
equipped with the norm
\[
\|Z\|_{\XX_T}
=\sup_{0\le t\le T}\|Z(t)\|_{L^2}
+\left(\int_0^T\|Z(t)\|_{H^1}^2\,\mathrm{d}t\right)^{1/2}.
\]
\begin{lemma}[Linear evolution estimate]\label{lem:linear}
The semigroup $S_\gamma(t)=\exp(tL_\gamma(D))$ satisfies
\begin{equation}\label{5.1}
\left\|
S_\gamma(t)Z_0
+\int_0^t S_\gamma(t-s)F(s)\,\mathrm{d}s
\right\|_{\XX_T}
\le C_\gamma\left(
\|Z_0\|_2+\|F\|_{L^2(0,T;H^{-1})}
\right)
\end{equation}
for $0<T\le1$, with $C_\gamma$ depending only on $\gamma$.
\end{lemma}
\begin{proof}
For smooth data and forcing with compact Fourier support, set $k=|\xi|$.
For $k\ge1$, write
\[
L_\gamma(\xi)=k^2A_c,\qquad
A_c=\begin{pmatrix}0&c\\-c&-\gamma\end{pmatrix},
\qquad
c=\frac{\sqrt{1+k^4}}{k^2}\in[1,\sqrt2].
\]
Define
\[
P_c=\begin{pmatrix}
\gamma^{-1}+\gamma/(2c^2)&(2c)^{-1}\\
(2c)^{-1}&\gamma^{-1}
\end{pmatrix}.
\]
A direct calculation gives
\[
A_c^*P_c+P_cA_c=-I.
\]
Moreover,
\[
\det P_c=\gamma^{-2}+\frac{1}{4c^2},
\qquad
\operatorname{tr}P_c=2\gamma^{-1}+\frac{\gamma}{2c^2}.
\]
Hence $P_c$ is uniformly positive definite for $c\in[1,\sqrt2]$. In particular, there exist constants $0<p_-\le p_+<\infty$, depending only on $\gamma$, such that
\[
p_-I\le P_c\le p_+I.
\]
Set $P(\xi)=P_c$ for $k\ge1$ and $P(\xi)=I$ for $k<1$,
and define
\[
\mathcal L(Z)
=\int_{\R}\widehat Z(\xi)^*P(\xi)\widehat Z(\xi)
\,\mathrm{d}\xi.
\]
For $k<1$, we have
\[
L_\gamma(\xi)^*+L_\gamma(\xi)
=\operatorname{diag}(0,-2\gamma k^2)\le0.
\]
For $k\ge1$, the identity $A_c^*P_c+P_cA_c=-I$ yields the dissipation term
$-k^2|\widehat Z|^2$. Consequently,
\[
\frac{\mathrm{d}}{\mathrm{d}t}\mathcal L(Z(t))
+\int_{k\ge1}k^2|\widehat Z|^2\,\mathrm{d}\xi
\le
2\left|\int_{\R}\widehat F^*P\widehat Z
\,\mathrm{d}\xi\right|
\le C_\gamma\|F\|_{H^{-1}}\|Z\|_{H^1},
\]
where we have used the uniform bound on $P$ and Cauchy-Schwarz with weights $\langle\xi\rangle^{-1}$ and $\langle\xi\rangle$ in the last inequality.
Using
\[
\|Z\|_{H^1}^2
\le2\|Z\|_{L^2}^2
+2\int_{k\ge1}k^2|\widehat Z|^2\,\mathrm{d}\xi
\]
and the equivalence $\mathcal L(Z)\simeq\|Z\|_{L^2}^2$, it follows from Young's inequality that
\[
\frac{\mathrm{d}}{\mathrm{d}t}\mathcal L(Z(t))
+c_\gamma\|Z\|_{H^1}^2
\le C_\gamma\mathcal L(Z(t))
+C_\gamma\|F\|_{H^{-1}}^2.
\]
Multiplying by an exponential integrating factor and integrating in time yields bounds for
$\sup_{0\le t\le T}\mathcal L(Z(t))$ and $\int_0^T\|Z(t)\|_{H^1}^2\,\mathrm{d}t$,
with constants depending only on $\gamma$ for $T\le1$. Taking square roots and using the equivalence of $\mathcal L(Z)$ and $\|Z\|_{L^2}^2$ proves \eqref{5.1} for smooth data.

To extend the estimate to general data, let $\mathbb{P}_n$ be the spatial Fourier projection onto $\{|\xi|\le n\}$, and let $\widetilde F$ denote the extension of $F$ by zero outside $(0,T)$. Choose a nonnegative function $\rho\in C_c^\infty(\R)$ with $\int_\R\rho=1$, and set $\rho_\varepsilon(t)=\varepsilon^{-1} \rho(t/\varepsilon)$. Define
\[
Z_{0,n}=\mathbb{P}_nZ_0,\qquad
F_n=\left.
\bigl(\rho_{1/n}*_t \mathbb{P}_n\widetilde F\bigr)
\right|_{[0,T]}.
\]
Then
\[
Z_{0,n}\to Z_0\quad\text{in }L^2,\qquad
F_n\to F\quad\text{in }L^2(0,T;H^{-1}).
\]
Moreover, $F_n\in C^\infty([0,T];L^2)$, and the spatial Fourier transforms of $Z_{0,n}$ and $F_n(t)$ are supported in $\{|\xi|\le n\}$ for every $t$.

On the range of $\mathbb{P}_n$, the operator $L_\gamma(D)$ is bounded on $L^2$. Hence the problem
\[
\partial_tZ_n=L_\gamma(D)Z_n+F_n,\qquad
Z_n(0)=Z_{0,n},
\]
has a unique solution in $C^1([0,T];L^2)$, with spatial Fourier support contained in $\{|\xi|\le n\}$. For almost every $\xi$, its Fourier transform is given by
\[
\widehat Z_n(t,\xi)
=e^{tL_\gamma(\xi)}\widehat Z_{0,n}(\xi)
+\int_0^t
e^{(t-s)L_\gamma(\xi)}
\widehat F_n(s,\xi)\,\mathrm ds.
\]
Applying the preceding estimate to $Z_n-Z_m$, we obtain
\[
\|Z_n-Z_m\|_{\XX_T}
\le C_\gamma\left(
\|Z_{0,n}-Z_{0,m}\|_{L^2}
+\|F_n-F_m\|_{L^2(0,T;H^{-1})}
\right).
\]
Thus $(Z_n)$ converges in $\XX_T$ to a limit $Z$, and convergence in $C([0,T];L^2)$ implies $Z(0)=Z_0$. Since $L_\gamma(D):H^1\to H^{-1}$ is bounded,
\[
L_\gamma(D)Z_n+F_n
\longrightarrow L_\gamma(D)Z+F
\quad\text{in }L^2(0,T;H^{-1}).
\]
Passing to the limit in the evolution equation
therefore yields
\[
Z_t=L_\gamma(D)Z+F
\]
in the sense of distributions, with $Z_t\in L^2(0,T;H^{-1})$. Estimate \eqref{5.1} follows by passage to the limit. For zero initial data, this construction defines
the Duhamel term for general $F\in L^2(0,T;H^{-1})$. The difference estimate shows that the resulting limit is independent of the approximating sequence.

For uniqueness, let $W$ be the difference of two distributional solutions in $\XX_T$ with the same data and forcing. Let
\[
\mathbb P_n=\mathbf 1_{\{|\xi|\le n\}}(D),
\qquad W_n=\mathbb P_nW.
\]
For each fixed $n$, $\mathbb P_n: H^{-1}\to L^2$ is bounded and commutes with $L_\gamma(D)$. Hence $W_n\in C^1([0,T];L^2)$ and
\[
\partial_tW_n=L_\gamma(D)W_n,\qquad
W_n(0)=0.
\]
Since $L_\gamma(D)$ is bounded on the range of $\mathbb P_n$, uniqueness for this ODE gives $W_n=0$. Letting $n\to\infty$ yields $W=0$. Thus every distributional solution in $\XX_T$ coincides with the corresponding mild solution. When $F=0$, the Fourier formula gives the semigroup property and strong continuity at $t=0$ for compactly supported Fourier data. Density and the uniform $L^2$ bound extend continuity to all $L^2$ data.
\end{proof}

\begin{proposition}[Local well-posedness]\label{prop:local}
Fix $\gamma>0$. For every $M>0$, there exists $T_*=T_*(\gamma,M)>0$ such that each
$(\eta_0,\psi_0)\in\Ecal$ satisfying
\[
\|\eta_0\|_{H^2}+\|\partial_x\psi_0\|_{L^2}\le M
\]
admits a unique solution in the class \eqref{1.12} on $[0,T_*]$. The associated variable $Z$ satisfies
\[
\|Z\|_{\XX_{T_*}}\le C_{\gamma,M}.
\]
No $L^2$ bound on $\psi_0$ enters the lifespan. On bounded subsets of $\Ecal$, the local solution map is Lipschitz into $C([0,T_*];\Ecal)$.
\end{proposition}
\begin{proof}
\emph{Initial chart and bounds.} Choose $f=\mathbb{P}\psi_0$ and
\begin{equation}\label{5.2}
v_0=(\Id-\mathbb{P})\psi_0+G(\eta_0)\psi_0.
\end{equation}
Writing $g_0=(\Id-\mathbb{P})\psi_0$, the cutoff estimate and Lemma~\ref{lem:homogeneous} imply
\[
\|g_0\|_{H^1}\le\sqrt2\|\partial_x\psi_0\|_{L^2},
\qquad
\|v_0\|_2+\|f_x\|_2
\le C_M\|\partial_x\psi_0\|_{L^2}.
\]
Since
\[
v_0-G(\eta_0)f=(\Id+G(\eta_0))g_0,
\]
equation \eqref{4.1} yields $\zeta(0)=g_0$ and hence $\psi(0)=\psi_0$. The reconstruction imposes no further compatibility condition.

\emph{Contraction.}
Let $A_M$ be a uniform bound for
$\|Z_0\|_2+\|f_x\|_2$, and choose
$R>2C_\gamma A_M+1$. For $0<T\le1$, define
\[
\Phi_f(Z)(t)
=S_\gamma(t)Z_0
+\int_0^t S_\gamma(t-s)\Fcal_f(Z(s))\,\mathrm{d}s
\]
on the closed ball of radius $R$ centered at zero in $\XX_T$. On this ball, the reconstruction $\eta=b(D)^{-1}Z_1$ gives
\[
\sup_{t\le T}\bigl(\|\eta(t)\|_{H^2}+\|v(t)\|_{L^2}\bigr)
+\|f_x\|_{L^2}
\le C R+A_M.
\]
Thus Corollary~\ref{cor:forcing} applies with uniform constants. Combining its bounds with \eqref{5.1} yields
\begin{align*}
\|\Phi_f(Z)\|_{\XX_T}
&\le C_\gamma A_M+C_{\gamma,R,M}\sqrt T,\\
\|\Phi_f(Z_1)-\Phi_f(Z_2)\|_{\XX_T}
&\le C_{\gamma,R,M}\sqrt T
\|Z_1-Z_2\|_{\XX_T}.
\end{align*}
Choose $0<T_*\le1$ such that
\[
C_{\gamma,R,M}\sqrt{T_*}\le\min\{R/2,1/2\}.
\]
Then $\Phi_f$ maps the complete ball into itself with Lipschitz constant at most $1/2$. The Banach fixed-point theorem gives a unique fixed point, with the choices depending only on $\gamma$ and $M$.

\emph{Reconstruction and the equations.} The boundedness of $b(D)^{-1}:L^2\to H^2$ and $b(D)^{-1}:H^1\to H^3$ gives $\eta\in C_T H^2\cap L_T^2H^3$. Lemma~\ref{lem:chart} yields $\zeta\in C_T H^1$, and hence $h=v-\zeta\in C_T L^2\cap L_T^2H^1$. The fixed-point equation implies \eqref{4.8} in the sense of distributions. Its components give
\[
\eta_t=v-\zeta=h,\qquad
v_t=\gamma h_{xx}-B\eta-\Ncal(\eta,f+\zeta).
\]
Together with \eqref{4.2}, these identities yield \eqref{1.8} for $u=f+v$ and $\psi=f+\zeta$.
Since $H^1\hookrightarrow\Pcal$ and $f\in\Pcal$, we have $\psi\in C_T\Pcal$.

\emph{Difference estimates.}
For two initial data sets satisfying the same bound, write $\delta Z=Z_1-Z_2$ and $\delta f=f_1-f_2$. Subtracting the mild formulas and using \eqref{5.1}
and \eqref{4.11}, we have
\[
\|\delta Z\|_{\XX_{T_*}}
\le C_\gamma\|\delta Z_0\|_{L^2}
+C_{\gamma,R,M}\sqrt{T_*}
\bigl(\|\delta Z\|_{\XX_{T_*}}
+\|\partial_x\delta f\|_{L^2}\bigr).
\]
If necessary, choose a smaller $T_*=T_*(\gamma,M)$ such that $C_{\gamma,R,M}\sqrt{T_*}\le1/2$. Absorbing the term involving $\|\delta Z\|_{\XX_{T_*}}$ into the left-hand side, it holds:
\begin{equation}\label{5.3}
\|Z_1-Z_2\|_{\XX_{T_*}}
\le C_{\gamma,M}\bigl(
\|Z_{0,1}-Z_{0,2}\|_{L^2}
+\|\partial_x(f_1-f_2)\|_{L^2}
\bigr).
\end{equation}
To estimate the initial chart variables, decompose
\[
G(\eta_{0,1})\psi_{0,1}-G(\eta_{0,2})\psi_{0,2}
=G(\eta_{0,1})(\psi_{0,1}-\psi_{0,2})
+\bigl(G(\eta_{0,1})-G(\eta_{0,2})\bigr)\psi_{0,2}.
\]
Using \eqref{5.2}, the homogeneous estimates, and the cutoff bounds, we obtain
\[
\|\delta Z_0\|_2+\|\partial_x\delta f\|_2
\le C_M\bigl(
\|\delta\eta_0\|_{H^2}
+\|\partial_x\delta\psi_0\|_2
\bigr).
\]
It follows from $H^1\hookrightarrow\Pcal$ and
$\psi_1-\psi_2=\delta f+\zeta_1-\zeta_2$ that
\[
\sup_{0\le t\le T_*}\|\psi_1(t)-\psi_2(t)\|_{\Pcal}
\le \|\delta f\|_{\Pcal}
+C\sup_{0\le t\le T_*}\|\zeta_1(t)-\zeta_2(t)\|_{H^1}.
\]
Combining Lemma~\ref{lem:chart}, \eqref{5.3}, and $\|\delta f\|_{\Pcal}\le\|\delta\psi_0\|_{\Pcal}$ proves the asserted Lipschitz dependence in
$C([0,T_*];\Ecal)$. Finally, $\delta h=\delta v-(\zeta_1-\zeta_2)$ satisfies
the corresponding difference estimate in $C_{T_*}L^2\cap L_{T_*}^2H^1$.

\emph{Uniqueness in the stated solution class.}
For any solution in the class \eqref{1.12}, the reconstructed variable satisfies $Z\in\XX_T$ and $\Fcal_f(Z)\in C_T H^{-1}$. Since $Z$ solves \eqref{4.8} distributionally,
Lemma~\ref{lem:linear} implies that it satisfies the corresponding mild formula. For two solutions with the same initial data, fix $f=\mathbb{P}\psi_0$. On a finite common interval, their $C_T L^2$ norms are bounded; the difference estimate therefore applies on subintervals of length $\tau$ with an absorption factor $C\sqrt{\tau}$. Here $C$ depends only on $\gamma$, $\|f_x\|_{L^2}$, and upper bounds for $\|Z_j\|_{C([0,T];L^2)}$, $j=1,2$. Choose $\tau$ so that $C\sqrt{\tau}\le1/2$ and iterate from the common initial value to obtain equality throughout the interval.
\end{proof}

\begin{remark}\label{rem5.3}
Let $f_1,f_2$ be admissible fixed backgrounds with $f_1-f_2\in H^1$. The change of variables
\[
v_2=v_1+f_1-f_2,\qquad
\zeta_2=\zeta_1+f_1-f_2
\]
preserves $u$, $\psi$, and $h$, as well as the reduced equations and the regularity class \eqref{1.12} on finite intervals. At a restart time $t_0$, choosing
$f_{t_0}=\mathbb{P}\psi(t_0)$ satisfies
\[
f_{t_0}-f=\mathbb{P}\zeta(t_0)\in H^1.
\]
Lemma~\ref{lem:affine-tools} verifies the compatibility of the local constructions at successive restart times.
\end{remark}

\section{Global well-posedness}\label{sec:global}
\subsection{Energy identity}\label{sec:energy}
\noindent

For a fixed background $f$, define $E_f(Z)=E(\eta,f+\zeta)$, with $\eta=b(D)^{-1}Z_1$, $v=Z_2$, and $\zeta$ given by \eqref{4.1}. This energy is finite for every $Z\in L^2(\R;\R^2)$.
\begin{lemma}[Differentiability of the energy functional]\label{lem:energy-gradient}
The functional $E_f:\HH\to\R$, with $\HH=L^2(\R;\R^2)$, is continuously Fr\'echet differentiable. Its $\HH$-gradient is
\begin{equation}\label{6.1}
\grad_ZE_f(Z)=Z+K_f(Z),\qquad
K_f(Z)=
\begin{pmatrix}
b(D)^{-1}\Ncal(\eta,f+\zeta)\\
-\zeta
\end{pmatrix}.
\end{equation}
The map $K_f:\HH\to H^1(\R;\R^2)$ is locally Lipschitz.
\end{lemma}
\begin{proof}
We first assume $f\in H^1$ and then remove this restriction.
Write
\[
G=G(\eta),\qquad R=(\Id+G)^{-1},\qquad A=GR=\Id-R.
\]
The operator $G$ is self-adjoint on $L^2$ with domain $H^1$, while $R$ and $A$ are bounded self-adjoint operators on $L^2$. Since $u=f+v\in L^2$, we have
$\psi=Ru\in H^1$ and $h=G\psi=Au$. Using $u=\psi+h$, we obtain
\[
\|h\|_2^2+\Gform_\eta(\psi,\psi)
=\langle h,h+\psi\rangle_{L^2}
=\langle Au,u\rangle_{L^2}.
\]
Thus the kinetic energy is $\tfrac12\langle u,Au\rangle_{L^2}$.

\emph{Variation of the graph.}
Hold $u\in L^2$ fixed. By Theorem~\ref{thm:B-shape}, the map $\eta\mapsto G(\eta)$ is continuously Fr\'echet differentiable from the real-valued subspace of $H^2$
into $\mathcal L(H^1,L^2)$. Moreover, Lemma~\ref{lem:elliptic-input} shows that
$R_\eta\in\mathcal L(L^2,H^1)$ and that $\eta\mapsto R_\eta$ is locally Lipschitz from $H^2$ into $\mathcal L(L^2,H^1)$. For a real-valued increment $k\in H^2$, write
\[
\widetilde G=G(\eta+k),\qquad
\widetilde R=(\Id+\widetilde G)^{-1},\qquad
Q_k=\widetilde G-G-DG(\eta)[k].
\]
It follows from the remainder estimate \eqref{B22} that
\[
\|Q_k\|_{\mathcal L(H^1,L^2)}
\le C_M\|k\|_{H^2}^2,
\]
whenever $\|\eta\|_{H^2},\|\eta+k\|_{H^2}\le M$ and $\|k\|_{H^2}\le1$.
The resolvent identity, interpreted in $\mathcal L(L^2,H^1)$, yields
\[
\widetilde R-R=-\widetilde R(\widetilde G-G)R.
\]
Consequently,
\[
\widetilde R-R+R\,DG(\eta)[k]R
=(R-\widetilde R)DG(\eta)[k]R-\widetilde RQ_kR.
\]
Combining \eqref{3.2}, \eqref{3.3}, and \eqref{3.6},
we obtain
\[
\|\widetilde R-R+R\,DG(\eta)[k]R
  \|_{\mathcal L(L^2,H^1)}
\le C_M\|k\|_{H^2}^2.
\]
Thus $\eta\mapsto R_\eta$ is Fr\'echet differentiable
with values in $\mathcal L(L^2,H^1)$, and
\[
DR_\eta[k]=-R_\eta DG(\eta)[k]R_\eta.
\]
Continuity of this derivative follows from the operator-norm continuity of $R_\eta$ and $DG(\eta)$.

Since $A(\eta)=\Id-R_\eta$, the map $\eta\mapsto A(\eta)$ is continuously Fr\'echet
differentiable with values in $\mathcal L(L^2,L^2)$, and its derivative satisfies the stronger mapping property
\[
DA(\eta)[k]
=R_\eta DG(\eta)[k]R_\eta
\in\mathcal L(L^2,H^1),\qquad
\|DA(\eta)[k]\|_{\mathcal L(L^2,H^1)}
\le C_M\|k\|_{H^2}.
\]
Using the self-adjointness of $R$ and $\psi=Ru$, it follows
\[
\frac12\langle u,DA(\eta)[k]u\rangle_{L^2}
=\frac12\langle\psi,DG(\eta)[k]\psi\rangle_{L^2}.
\]
Let
\[
\beta=\frac{h+\eta_x\psi_x}{1+\eta_x^2},
\qquad
\nu=\psi_x-\eta_x\beta.
\]
Both $\beta$ and $\nu$ belong to $L^2$. Testing \eqref{3.4} against $\psi\in H^1$ and using $G\psi=h$, we obtain
\begin{align*}
\frac12\langle\psi,DG(\eta)[k]\psi\rangle_{L^2}
&=\frac12\int_{\R}k(-h\beta+\psi_x\nu)\,\mathrm{d}x\\
&=\frac12\int_{\R}
k\bigl(\psi_x^2-\beta(h+\eta_x\psi_x)\bigr)
\,\mathrm{d}x\\
&=\int_{\R}k\Ncal(\eta,\psi)\,\mathrm{d}x.
\end{align*}
These pairings are well defined. By self-adjointness, $G:H^1\to L^2$ admits a bounded extension $G:L^2\to H^{-1}$ defined by duality. Since $k\beta,k\nu\in L^2$, both terms in the distributional identity belong to $H^{-1}$. Density therefore extends the identity to test functions in $H^1$, including $\psi$.

\emph{Variation of the velocity variable.}
At a fixed graph, it follows from self-adjointness of $A$ that
\[
D_v\!\left(\tfrac12\langle u,Au\rangle_{L^2}\right)[w]
=\langle Au,w\rangle_{L^2}
=\langle h,w\rangle_{L^2},
\qquad w\in L^2.
\]
Write $\widetilde E_f(\eta,v)=E_f(b(D)\eta,v)$. Combining the graph and velocity variations with the derivative of the structural energy, we obtain
\begin{equation}\label{6.2}
D\widetilde E_f(\eta,v)[k,w]
=\langle\eta,k\rangle_{L^2}
+\langle\eta_{xx},k_{xx}\rangle_{L^2}
+\int_{\R}\Ncal k\,\mathrm{d}x
+\langle h,w\rangle_{L^2}.
\end{equation}
In the $Z$ coordinates, a direction $(a,w)$ has $k=b(D)^{-1}a$, so the corresponding gradient is
\[
\bigl(b(D)\eta+b(D)^{-1}\Ncal,h\bigr),
\]
which is \eqref{6.1}. The bound $b(D)^{-1}:H^{-1}\to H^1$, Lemma~\ref{lem:chart}, and $L^1\hookrightarrow H^{-1}$ show that $K_f$ is locally Lipschitz into $H^1$.
Integrating the directional derivative along the segment from $Z$ to $Z+W$ gives
\[
\bigl|E_f(Z+W)-E_f(Z)
-\langle Z+K_f(Z),W\rangle_{L^2}\bigr|
\le C\|W\|_{L^2}^2
\]
for sufficiently small $W$. This proves continuous Fr\'echet differentiability.

\emph{Homogeneous background.}
For a general background $f\in\Pcal$ with
$\supp\widehat f\subset[-1,1]$, define
\[
\widehat f_n
=\mathbf1_{\{|\xi|\ge1/n\}}\widehat f.
\]
Then $f_n\in H^1$, $f_n\to f$ in $\Pcal$, and $\partial_xf_n\to\partial_xf$ in $L^2$.
For each $Z=(b(D)\eta,v)$, set
\[
\zeta_n=R_\eta(v-G(\eta)f_n),\qquad
\psi_n=f_n+\zeta_n,\qquad h_n=v-\zeta_n.
\]
Fix $R_0>0$, suppose that $\sup_{\|Z\|_{\HH}}\le R_0$, and set $d_n=f_n-f$. Since  $\|f_n\|_{\mathcal P}\le\|f\|_{\mathcal P}$, the constants below are independent of $n$ and of $Z$ in this ball. By \eqref{3.2} and \eqref{3.7}, it follows
\[
\zeta_n-\zeta=-R_\eta G(\eta)d_n,\qquad
\|\zeta_n-\zeta\|_{H^1}+\|h_n-h\|_{L^2}
\le C_{R_0}\|\partial_xd_n\|_{L^2}.
\]
Consequently,
\[
\|\psi_n-\psi\|_{\dot H^{1/2}}
\le \|d_n\|_{\dot H^{1/2}}+C_{R_0}\|\partial_xd_n\|_{L^2}.
\]
The bounds \eqref{4.3} and $\|f_n\|_{\mathcal P}\le\|f\|_{\mathcal P}$ give uniform bounds for $h_n$ in $L^2$ and $\psi_n$ in $\dot H^{1/2}$. The structural energy terms cancel because both energies are evaluated at the same graph. Applying \eqref{3.14} to the fluid energy and the preceding $L^2$ estimate to the beam kinetic energy yields
\[
\sup_{\|Z\|_\HH\le R_0}|E_{f_n}(Z)-E_f(Z)|
\le C_{R_0,f}\|d_n\|_{\mathcal P}\longrightarrow0.
\]
Moreover, \eqref{4.6} gives
\(\|\mathcal N(\eta,\psi_n)-\mathcal N(\eta,\psi)\|_{L^1}
\le C_{R_0,f}\|\partial_xd_n\|_{L^2}\).
Thus $L^1\hookrightarrow H^{-1}$, $b(D)^{-1}:H^{-1}\to H^1$, and \eqref{6.1} imply
\[
\sup_{\|Z\|_\HH\le R_0}\|K_{f_n}(Z)-K_f(Z)\|_{H^1}
\le C_{R_0,f}\|\partial_xd_n\|_{L^2}\longrightarrow0.
\]
Passing to the limit in
\[
E_{f_n}(Z+W)-E_{f_n}(Z)
=\int_0^1
\langle Z+\theta W+K_{f_n}(Z+\theta W),W\rangle_{L^2}
\,\mathrm{d}\theta
\]
yields the corresponding identity for $f$. Since $Z\mapsto Z+K_f(Z)$ is locally Lipschitz in $L^2$, it follows that $E_f\in C^1(\HH;\R)$ with gradient \eqref{6.1}, even when $f\notin L^2$. Lemma~\ref{lem:chart} and the multiplier bound above also give local Lipschitz continuity of $K_f$ into $H^1$.
\end{proof}

The preceding lemma identifies the $L^2$-gradient of $E_f$ and shows that its correction term $K_f$ takes values in $H^1$. To justify differentiation along local solutions with $Z_t\in L^2(0,T;H^{-1})$, we next establish a chain rule on the Gelfand triple $H^1\subset L^2\subset H^{-1}$.
\begin{lemma}\label{lem:chain}
Let
\[
\VV=H^1(\R;\R^2)\subset\HH=L^2(\R;\R^2)
\subset\dual=H^{-1}(\R;\R^2).
\]
Suppose that $\mathcal J\in C^1(\HH;\R)$ has $\HH$-gradient
\[
\grad\mathcal J(Z)=Z+K(Z),
\]
where $K:\HH\to\VV$ is locally Lipschitz. If
\[
Z\in C([0,T];\HH)\cap L^2(0,T;\VV),
\qquad
Z_t\in L^2(0,T;\dual),
\]
then $t\mapsto\mathcal J(Z(t))$ is absolutely continuous on $[0,T]$ and satisfies
\begin{equation}\label{6.3}
\mathcal J(Z(t))-\mathcal J(Z(s))
=\int_s^t
\ip{Z_\tau(\tau)}{\grad\mathcal J(Z(\tau))}_{\dual,\VV}
\,\mathrm{d}\tau,
\qquad 0\le s\le t\le T.
\end{equation}
\end{lemma}
\begin{proof}
Fix $[a,b]\subset(0,T)$ and let $\rho\in C_c^\infty((-1,1))$ be a standard mollifier.
For $0<\varepsilon<\min\{a,T-b\}$, define $Z^\varepsilon=\rho_\varepsilon*_t Z$ on $[a,b]$. Then $Z^\varepsilon$ is smooth as an $\HH$-valued function,
and
\[
Z^\varepsilon\to Z
\quad\text{in }C([a,b];\HH)\cap L^2(a,b;\VV),
\qquad
\partial_tZ^\varepsilon\to Z_t
\quad\text{in }L^2(a,b;\dual).
\]
Uniform continuity of $Z:[0,T]\to\HH$ yields $Z^\varepsilon\to Z$ in $C([a,b];\HH)$, while
the approximation property of time convolution gives convergence in $L^2(a,b;\VV)$.
Since $[a,b]\subset(0,T)$, convolution commutes with the distributional time derivative for sufficiently small $\varepsilon$, so that
\[
\partial_tZ^\varepsilon=\rho_\varepsilon*_t Z_t
\quad\text{on }(a,b).
\]
Applying the same approximation property to
$Z_t\in L^2(0,T;\dual)$ yields
$\partial_tZ^\varepsilon\to Z_t$
in $L^2(a,b;\dual)$.

Since $Z([a,b])$ is compact in $\HH$, a finite covering by neighborhoods on which $K$ is Lipschitz gives
\[
\sup_{t\in[a,b]}
\|K(Z^\varepsilon(t))-K(Z(t))\|_{\VV}
\le C\sup_{t\in[a,b]}
\|Z^\varepsilon(t)-Z(t)\|_{\HH}
\]
for sufficiently small $\varepsilon$. Thus $K(Z^\varepsilon)\to K(Z)$ uniformly in $\VV$. Together with the convergence of $Z^\varepsilon$ in $L^2(a,b;\VV)$, this yields
\[
\nabla\mathcal J(Z^\varepsilon)
\to\nabla\mathcal J(Z)
\quad\text{in }L^2(a,b;\VV).
\]
Applying the Banach-space chain rule to $Z^\varepsilon$, we have
\[
\mathcal J(Z^\varepsilon(b))
-\mathcal J(Z^\varepsilon(a))
=\int_a^b
\langle\partial_tZ^\varepsilon,
\nabla\mathcal J(Z^\varepsilon)\rangle_{\dual,\VV}
\,\mathrm{d}t.
\]
Here the $\HH$ inner product agrees with the $\dual$--$\VV$ pairing. The $L^1(a,b)$ norm of the difference between the integrands is bounded by
\[
\|\partial_tZ^\varepsilon-Z_t\|_{L^2(a,b;\dual)}
 \|\nabla\mathcal J(Z^\varepsilon)\|_{L^2(a,b;\VV)}
+\|Z_t\|_{L^2(a,b;\dual)}
\|\nabla\mathcal J(Z^\varepsilon)
-\nabla\mathcal J(Z)\|_{L^2(a,b;\VV)},
\]
which tends to zero. Continuity of $\mathcal J$ on $\HH$ gives convergence of the endpoint values, so passage to the limit proves the identity on $[a,b]$.

Finally, continuity of $K:\HH\to\VV$ and compactness of $Z([0,T])$ bound $K(Z(t))$ in $\VV$. Thus $\nabla\mathcal J(Z)\in L^2(0,T;\VV)$ and its pairing with $Z_t$ lies in $L^1(0,T)$. Continuity extends \eqref{6.3} to all $0\le s\le t\le T$, and the integral representation gives absolute continuity of $t\mapsto\mathcal J(Z(t))$.
\end{proof}

\begin{proposition}\label{prop:energy}
Every local solution of Proposition~\ref{prop:local} satisfies the energy identity \eqref{1.14} on its interval of existence.
\end{proposition}
\begin{proof}
Fix a finite interval $[0,T]$ within the existence interval of the
local solution, and retain the time-independent background $f$ used in
its construction. Write $\psi=f+\zeta$, $h=v-\zeta=\eta_t$, and
$\Ncal=\Ncal(\eta,\psi)$. By the solution class,
\[
Z=(b(D)\eta,v)\in C_TL^2\cap L_T^2H^1.
\]
Both $b(D)$ and $\partial_x^2$ map $H^1$ continuously into $H^{-1}$.
Hence $L_\gamma(D)Z\in L_T^2H^{-1}$. Since
$\Fcal_f(Z)\in C_TH^{-1}\subset L_T^2H^{-1}$, equation \eqref{4.8}
implies $Z_t\in L_T^2H^{-1}$. Lemma~\ref{lem:energy-gradient}
establishes that $E_f\in C^1(L^2;\R)$ and that its gradient has the
form required by Lemma~\ref{lem:chain}. The latter lemma therefore
shows that
\[
E(t):=E(\eta(t),\psi(t))=E_f(Z(t))
\]
is absolutely continuous on $[0,T]$ and satisfies the stated
$H^{-1}$-$H^1$ chain rule.

By \eqref{6.1} and $Z=(b(D)\eta,v)$, it holds that
\[
\nabla_ZE_f(Z)
=\begin{pmatrix}
b(D)\eta+b(D)^{-1}\Ncal\\
v-\zeta
\end{pmatrix}
=\begin{pmatrix}
b(D)\eta+b(D)^{-1}\Ncal\\
h
\end{pmatrix}.
\]
Moreover, \eqref{4.7} gives, in $H^{-1}$ for almost every $t$,
\[
Z_t=
\begin{pmatrix}
b(D)h\\
\gamma h_{xx}-B\eta-\Ncal
\end{pmatrix}.
\]
The symbol $b(\xi)$ is real and even, and $B=b(D)^2$. For almost every time, $h\in H^1$, $\eta\in H^3$, and $\Ncal\in H^{-1}$, so all the following duality pairings are well defined. The symmetry of the Fourier multiplier $b(D)$, together with $b(D)^2=B$, gives
\[
\begin{aligned}
\langle b(D)h,b(D)\eta\rangle_{H^{-1},H^1}
  &=\langle B\eta,h\rangle_{H^{-1},H^1},\\
\langle b(D)h,b(D)^{-1}\Ncal\rangle_{H^{-1},H^1}
  &=\langle\Ncal,h\rangle_{H^{-1},H^1}.
\end{aligned}
\]
These identities hold first for real-valued Schwartz functions and extend to the stated Sobolev spaces by density and continuity of the multipliers and duality pairings.
Consequently, the chain rule yields, for almost every $t\in(0,T)$,
\begin{align*}
E'(t)
&=\bigl\langle b(D)h,\,
       b(D)\eta+b(D)^{-1}\Ncal
  \bigr\rangle_{H^{-1},H^1}
  +\bigl\langle\gamma h_{xx}-B\eta-\Ncal,\,
       h\bigr\rangle_{H^{-1},H^1}\\
&=\langle B\eta+\Ncal,h\rangle_{H^{-1},H^1}
  +\langle\gamma h_{xx}-B\eta-\Ncal,h\rangle_{H^{-1},H^1}\\
&=\gamma\langle h_{xx},h\rangle_{H^{-1},H^1}\\
&=-\gamma\|h_x\|_{L^2}^2.
\end{align*}
The pairings are integrable in time: $\eta\in L_T^2H^3$ implies $B\eta\in L_T^2H^{-1}$, while $\Ncal\in C_TH^{-1}$ and $h\in L_T^2H^1$. The last equality follows from distributional integration by parts. Integrating the almost-everywhere identity and
using the absolute continuity of $E$, we obtain
\[
E(t)+\gamma\int_s^t\|h_x(\tau)\|_{L^2}^2\,\mathrm d\tau=E(s),
\qquad 0\le s\le t\le T.
\]
Thus \eqref{1.14} is proved.
\end{proof}

\subsection{Global continuation}\label{sec:Continuation}
\begin{lemma}\label{lem:coercivity}
For $\|\eta\|_{H^2}\le M$ and $\psi\in\Pcal$, we have
\begin{equation}\label{6.4}
E(\eta,\psi)\simeq
\|\eta\|_{H^2}^2+\|\psi\|_{\Pcal}^2.
\end{equation}
Consequently, any solution with $E(t)\le E_0$ obeys \eqref{1.15} throughout its interval of existence.
\end{lemma}
\begin{proof}
For every $\xi\in\R$,
\[
\tfrac12(1+\xi^2)^2\le1+\xi^4\le(1+\xi^2)^2.
\]
Thus $\|\eta\|_{L^2}^2+\|\eta_{xx}\|_{L^2}^2$ is equivalent to $\|\eta\|_{H^2}^2$. For $\|\eta\|_{H^2}\le M$, it follows from Lemmas~\ref{lem:homogeneous} and~\ref{lem:form} that
\[
\|G(\eta)\psi\|_{L^2}^2\simeq\|\psi_x\|_{L^2}^2,
\qquad
\Gform_\eta(\psi,\psi)
\simeq\|\psi\|_{\dot H^{1/2}}^2.
\]
These comparisons prove \eqref{6.4}.

If $E(t)\le E_0$, the nonnegative energy terms give
\[
\|\eta(t)\|_{H^2}\le2\sqrt{E_0},
\qquad
\|h(t)\|_{L^2}\le\sqrt{2E_0}.
\]
Choosing $M=1+2\sqrt{E_0}$ in the fixed-graph estimates, we have
\[
\|\psi_x(t)\|_{L^2}^2+\|\psi(t)\|_{\dot H^{1/2}}^2
\le C_M E_0
\]
uniformly in time. This proves \eqref{1.15}. The upper bound in \eqref{6.4} also shows that bounded subsets of $\Ecal$ have uniformly bounded energy, as required for stability.
\end{proof}

We now begin to prove Theorem~\ref{thm:main}.

\begin{proof}[Proof of Theorem~\ref{thm:main}]
Proposition~\ref{prop:local} gives a unique local solution whose $\Ecal$ norm is bounded in terms of $E_0$ by Proposition~\ref{prop:energy} and Lemma~\ref{lem:coercivity}. Let $T_{\max}$ denote its maximal forward existence time.

For $0\le t_0<T_{\max}$, set
\[
f_{t_0}=\mathbb{P}\psi(t_0),\qquad
v_{t_0}=(\Id-\mathbb{P})\psi(t_0)+h(t_0).
\]
The cutoff estimates and the energy bound imply
\begin{equation}\label{6.5}
\|\eta(t_0)\|_{H^2}
+\|\partial_x f_{t_0}\|_{L^2}
+\|v_{t_0}\|_{L^2}
\le C(E_0).
\end{equation}
Since $\|\eta(t_0)\|_{H^2}+\|\psi_x(t_0)\|_2$ is uniformly bounded, Proposition~\ref{prop:local} provides a lifespan $\tau_*=\tau_*(\gamma,E_0)>0$ independent of $t_0$. Remark~\ref{rem5.3} places the restarted and original solutions in a common affine class, so uniqueness identifies them on their overlap. If $T_{\max}<\infty$, choose $t_0<T_{\max}$ with $t_0+\tau_*>T_{\max}$. The solution restarted at $t_0$ then extends beyond $T_{\max}$, contradicting maximality. Hence $T_{\max}=\infty$.

The backgrounds chosen at consecutive continuation steps differ by an $H^1$ function. Since the local existence time has a uniform positive lower bound, the solution extends
across every finite time interval after finitely many such steps. Consequently, on every finite time interval, the global solution belongs to the class \eqref{1.12} relative to
the fixed background $f=\mathbb{P}\psi_0$. Summing the energy identities over successive continuation intervals yields \eqref{1.14} for all $0\le s\le t<\infty$; combined with
Lemma~\ref{lem:coercivity}, this implies \eqref{1.15}. Successive application of local uniqueness on these intervals establishes global uniqueness within the class \eqref{1.12}.

Finally, let $\mathcal B\subset\Ecal$ satisfy $\sup_{U_0\in\mathcal B}\|U_0\|_{\Ecal}\le M$. By Lemma~\ref{lem:coercivity}, there exists $E_*=E_*(M)$ such that $E(U_0)\le E_*$
for every $U_0\in\mathcal B$. The energy identity and the coercivity estimate provide
uniform bounds on the solution norms required by the local theory. Consequently, the local existence time can be chosen as $\tau_*=\tau_*(\gamma,E_*)>0$ at every continuation step, with local Lipschitz constants depending only on $\gamma$ and $E_*$. For $T>0$, choose $0<\tau\le\tau_*$, let $L\ge1$ be a common local Lipschitz constant in $\Ecal$, and set
\[
J=\lceil T/\tau\rceil,\qquad
t_j=\min\{j\tau,T\},\qquad 0\le j\le J.
\]
For two solutions, define
\[
d_j=\|\eta_1(t_j)-\eta_2(t_j)\|_{H^2}
+\|\psi_1(t_j)-\psi_2(t_j)\|_{\Pcal}.
\]
Applying the local stability estimate at each restart gives
\[
\sup_{t_j\le t\le t_{j+1}}
\bigl(
\|\eta_1(t)-\eta_2(t)\|_{H^2}
+\|\psi_1(t)-\psi_2(t)\|_{\Pcal}
\bigr)
\le Ld_j.
\]
Since $d_{j+1}\le Ld_j$, iteration bounds the supremum by $L^Jd_0$. Hence
\begin{align}\label{6.6}
&\sup_{0\le t\le T}\bigl(
\|\eta_1(t)-\eta_2(t)\|_{H^2}
+\|\psi_1(t)-\psi_2(t)\|_{\Pcal}
\bigr)\notag\\
&\quad\le C_{\gamma,M,T}\bigl(
\|\eta_{0,1}-\eta_{0,2}\|_{H^2}
+\|\psi_{0,1}-\psi_{0,2}\|_{\Pcal}
\bigr).
\end{align}
Thus the solution map is Lipschitz on every finite time interval.
\end{proof}

\begin{remark}
The natural phase space admits potentials outside $L^2$. The linear estimate \eqref{5.1} depends on structural damping, and its constants need not remain bounded as $\gamma\downarrow0$. The argument therefore gives no undamped limit theorem.

The global energy identity yields
\[
\int_0^\infty\|\partial_x\eta_t(t)\|_2^2\,\mathrm{d}t
\le\frac{E_0}{\gamma},
\qquad
E(t)\downarrow E_\infty\ge0.
\]
The identity alone does not establish $E_\infty=0$ or an algebraic decay rate. Since $u(t)$ may lie outside $L^2$, Theorem~\ref{thm:main} also does not yield decay estimates requiring a uniform $L^2$ bound on $u$.
\end{remark}

\appendix
\section{Compatibility under changes of background}\label{app:tools}
\begin{lemma}[Compatibility of affine changes]\label{lem:affine-tools}
Let $f_1,f_2\in\Pcal$ be fixed backgrounds with $f_1-f_2\in H^1$. For $\eta\in H^2$ and $v_1\in L^2$, define
\[
\zeta_1=R_\eta(v_1-G(\eta)f_1),\qquad
\delta=f_1-f_2,\qquad
v_2=v_1+\delta,\qquad
\zeta_2=\zeta_1+\delta.
\]
Then
\[
\zeta_2=R_\eta(v_2-G(\eta)f_2),
\]
and both reconstructions yield the same $u$, $\psi$, and $h$; consequently, the reduced formulations associated with $f_1$ and $f_2$ are equivalent, and the change of background preserves the regularity class \eqref{1.12} on every finite time interval.
\end{lemma}
\begin{proof}
Write $G=G(\eta)$. All terms in the following identity belong to $L^2$:
\[
(\Id+G)\zeta_2
=v_1-Gf_1+\delta+G\delta
=v_2-Gf_2.
\]
Applying $R_\eta$ gives the reconstruction formula, while the definitions imply
\[
f_2+v_2=f_1+v_1,\qquad
f_2+\zeta_2=f_1+\zeta_1,\qquad
v_2-\zeta_2=v_1-\zeta_1.
\]
Hence $u$, $\psi$, and $h$, together with the first reduced equation, are unchanged; because $\delta$ is time independent, both sides of the second equation acquire $-\gamma\delta_{xx}$, so that equation also holds in distributions.
On every finite interval, $\delta\in H^1$ belongs to $C_T H^1\cap L_T^2H^1$, which preserves \eqref{1.12} under the change of background. For the low frequency backgrounds used in the construction,
$\mathbb{P} f_1=f_1$. At a restart, choosing $f_2=\mathbb{P}\psi(t_0)$ therefore yields
\[
f_2-f_1=\mathbb{P}\zeta_1(t_0)\in H^1.
\]
\end{proof}

\section{Elliptic theory on an \texorpdfstring{$H^2$}{H2} graph}\label{app:elliptic}
\noindent

This appendix establishes the fixed-graph estimates and shape differentiability used in Lemma~\ref{lem:elliptic-input} in one horizontal dimension. The proofs combine Fourier and variational estimates with a conormal-flux trace bound. Related results appear in Sections~3.1 and~3.3 and Appendices~C and~D of~\cite{ASY}; see also the treatment of the Dirichlet-Neumann operator in~\cite{Lannes}.

Write $\Pi=\R\times(-\infty,0)$, $J=(1-\partial_x^2)^{1/2}$, and
\[
\|F\|_{\mathcal Y^s}^2
=\int_{-\infty}^0
\|F(\cdot,z)\|_{H^s(\R)}^2\,\mathrm{d}z.
\]
For vector and matrix-valued functions, the squared norm is the sum of the squared component norms. Constants denoted by $C_M$ may vary from line to line and depend only on the common bound $M$ for the $H^2$ norms of the graphs under consideration.

\begin{lemma}[Commutator estimates]\label{lem:B-comm}
Let $a\in H^1(\R)$. For $0<s<1/2$ and $g\in L^2(\R)$,
\begin{equation}\label{B1}
\|[J^s,a]g\|_{L^2}
\le C_s\|a\|_{H^1}\|g\|_{L^2}.
\end{equation}
For $s=1/2$ and $g\in H^{3/8}$,
\begin{equation}\label{B2}
\|[J^{1/2},a]g\|_2
\le C\|a\|_{H^1}\|g\|_{H^{3/8}}.
\end{equation}
Consequently, for $g\in H^{1/2}$,
\begin{equation}\label{B3}
\|ag\|_{H^{1/2}}
\le C\|a\|_{H^1}\|g\|_{H^{1/2}}.
\end{equation}
The commutator estimates remain valid with $J^s$ replaced by $T_{s,N}=J^s\chi(D/N)$, with constants independent of $N\ge1$. Here $\chi$ is a fixed real-valued, even function in $C_c^\infty(\R)$ that equals one near zero.
\end{lemma}
\begin{proof}
For $0<s<1$, the symbol $m(\xi)=\langle\xi\rangle^s$ satisfies
\[
|m(\xi)-m(\zeta)|\le C_s|\xi-\zeta|^s.
\]
The Fourier representation of the commutator therefore yields
\[
\bigl|\widehat{[J^s,a]g}\bigr|
\le C_s\bigl(|\cdot|^s|\widehat a|\bigr)*|\widehat g|.
\]
For $0<s<1/2$, it follows from the Cauchy-Schwarz inequality that
\[
\||\xi|^s\widehat a\|_{L^1_\xi}
\le
\left(\int_{\R}\frac{|\xi|^{2s}}{1+\xi^2}
\,\mathrm{d}\xi\right)^{1/2}\|a\|_{H^1}
\le C_s\|a\|_{H^1}.
\]
Young's convolution inequality $L^1*L^2\to L^2$, followed by Plancherel's theorem, therefore yields \eqref{B1}.

For $s=1/2$, use Young's convolution inequality $L^{4/3}*L^{4/3}\to L^2$ instead. It follows from H\"older's inequality with $3/4=1/2+1/4$ that
\begin{align*}
\||\xi|^{1/2}\widehat a\|_{4/3}
&\le
\||\xi|^{1/2}\langle\xi\rangle^{-1}\|_4
\|a\|_{H^1},\\
\|\widehat g\|_{4/3}
&\le
\|\langle\xi\rangle^{-3/8}\|_4
\|g\|_{H^{3/8}}.
\end{align*}
Both weight norms are finite; combined with Plancherel's theorem, these estimates prove \eqref{B2}. For smooth $a$ and $g$, we combine
\[
J^{1/2}(ag)=aJ^{1/2}g+[J^{1/2},a]g
\]
with \eqref{B2} to obtain
\[
\|ag\|_{H^{1/2}}
\le \|a\|_\infty\|g\|_{H^{1/2}}
+C\|a\|_{H^1}\|g\|_{H^{3/8}}
\le C\|a\|_{H^1}\|g\|_{H^{1/2}},
\]
where we have used $H^1(\R)\hookrightarrow L^\infty(\R)$ and
$H^{1/2}(\R)\hookrightarrow H^{3/8}(\R)$ in the last inequality. A density argument then extends this estimate to all $a\in H^1(\R)$ and $g\in H^{1/2}(\R)$, proving \eqref{B3}.

For $0<s<1$, the truncated symbol $m_N(\xi)=\langle\xi\rangle^s\chi(\xi/N)$ satisfies
\[
|m_N(\xi)|\le C\langle\xi\rangle^s,\qquad
|m_N'(\xi)|\le C\langle\xi\rangle^{s-1},
\]
with constants independent of $N\ge1$. If $|\xi-\zeta|\le\langle\zeta\rangle/2$, the mean value theorem gives
\[
|m_N(\xi)-m_N(\zeta)|
\le C|\xi-\zeta|\langle\zeta\rangle^{s-1}
\le C|\xi-\zeta|^s.
\]
Otherwise, $\langle\zeta\rangle<2|\xi-\zeta|$ and $\langle\xi\rangle\le3|\xi-\zeta|$, so the size estimate gives the same bound; the preceding convolution arguments therefore apply uniformly in $N$. Density then extends them from smooth functions to the stated Sobolev spaces.
\end{proof}

\begin{lemma}[Normal trace of a divergence-free flux]\label{lem:B-flux}
Suppose that $Q=(Q_x,Q_z)\in\mathcal Y^{1/2}$ satisfies
\[
\partial_xQ_x+\partial_zQ_z=0
\]
in the sense of distributions in $\Pi$. Then $Q_z$ has a strong $L^2$ trace at $z=0$, satisfying
\begin{equation}\label{B4}
\|Q_z(0)\|_{L^2}\le C\|Q\|_{\mathcal Y^{1/2}}.
\end{equation}
More generally, if $q\in\mathcal Y^{1/2}$ and its distributional derivative $\partial_zq$ belongs to $\mathcal Y^{-1/2}$, then $q$ admits a representative in
$C((-\infty,0];L^2)$, and
\begin{equation}\label{B5}
\sup_{z\le0}\|q(z)\|_{L^2}^2
\le2\|q\|_{\mathcal Y^{1/2}}
\|\partial_zq\|_{\mathcal Y^{-1/2}}.
\end{equation}
\end{lemma}
\begin{proof}
For smooth $q$ satisfying the stated assumptions, $q\in L^2_zL^2_x$ permits a sequence $z_n\to-\infty$ such that $\|q(z_n)\|_2\to0$. Integrating
\[
\frac{\mathrm{d}}{\mathrm{d}z}\|q(z)\|_{L^2}^2
=2\operatorname{Re}
\langle\partial_zq(z),q(z)\rangle_{H^{-1/2},H^{1/2}}
\]
from $z_n$ to $z$, letting $n\to\infty$, and applying Cauchy-Schwarz gives \eqref{B5}. For general $q$, choose a nonnegative $\rho\in C_c^\infty((1,2))$ with $\int\rho=1$, and let $\mathbb{P}_n$ be the tangential Fourier projection onto $|\xi|\le n$. Set
\[
q_n(\cdot,z)=\int_1^2\rho(r)\mathbb{P}_nq(\cdot,z-r/n)\,\mathrm dr,
\qquad z\le0.
\]
This formula samples only points strictly inside $\Pi$ and defines a smooth $L^2$-valued function up to $z=0$. Differentiation commutes with this interior convolution in the distributional sense, so
\[
\partial_zq_n(\cdot,z)
=\int_1^2\rho(r)\mathbb{P}_n(\partial_zq)(\cdot,z-r/n)\,\mathrm dr.
\]
Strong continuity of translations in $L^2_zH^s_x$, applied after zero extension only to verify norm convergence, and strong convergence of $\mathbb{P}_n$ give
\[
q_n\to q\quad\text{in }\mathcal Y^{1/2},
\qquad
\partial_zq_n\to\partial_zq
\quad\text{in }\mathcal Y^{-1/2}.
\]
Applying the smooth estimate to $q_n-q_m$ shows that $(q_n)$ is uniformly Cauchy on $(-\infty,0]$ with values in $L^2$; its continuous limit agrees with $q$ almost everywhere and provides the claimed representative and trace, while passage to the limit preserves \eqref{B5}. The divergence-free condition further gives
$\partial_zQ_z=-\partial_xQ_x$, and hence
\[
\|\partial_zQ_z\|_{\mathcal Y^{-1/2}}
\le\|Q_x\|_{\mathcal Y^{1/2}}.
\]
Thus \eqref{B5}, applied to $Q_z$, yields \eqref{B4}.
\end{proof}

Let
\[
\mathscr C(\overline\Pi)
=\{\varphi|_\Pi:\varphi\in C_c^\infty(\R^2)\},
\]
and define $\dot{\mathcal V}(\Pi)$ as the completion of $\mathscr C(\overline\Pi)$ under the norm $\|\nabla w\|_{L^2(\Pi)}$. The Poincar\'e inequality on bounded half-rectangles realizes its elements as locally $H^1$ functions up to the boundary, modulo
additive constants. The trace construction below gives a continuous map
\[
\operatorname{Tr}:
\dot{\mathcal V}(\Pi)\longrightarrow\dot H^{1/2}(\R).
\]
We set
\[
\mathcal V_0
=\ker(\operatorname{Tr}),
\]
which is a closed subspace of $\dot{\mathcal V}(\Pi)$. For elements of $\mathcal V_0$, the additive constant is fixed by requiring the local boundary trace to vanish.

For $\eta\in H^2$, set $s=\eta_x$ and
\begin{equation}\label{B6}
A_\eta=
\begin{pmatrix}
1 & -s\\
-s & 1+s^2
\end{pmatrix}.
\end{equation}
\begin{lemma}[Variational construction and tangential regularity]
\label{lem:B-shift}
Let $A=A_\eta$, where $\|\eta\|_{H^2}\le M$. For every $F\in L^2(\Pi;\R^2)$, there exists a unique $w\in\mathcal V_0$ satisfying
\begin{equation}\label{B7}
\int_\Pi A\nabla w\cdot\nabla\theta
\,\mathrm{d}x\,\mathrm{d}z
=\int_\Pi F\cdot\nabla\theta
\,\mathrm{d}x\,\mathrm{d}z,
\qquad \theta\in\mathcal V_0.
\end{equation}
Moreover,
\[
\|\nabla w\|_{L^2(\Pi)}
\le C_M\|F\|_{L^2(\Pi)}.
\]
If $F\in\mathcal Y^{1/2}$, then
\begin{equation}\label{B8}
\|\nabla w\|_{\mathcal Y^{1/2}}
\le C_M\|F\|_{\mathcal Y^{1/2}}.
\end{equation}

For every $\psi\in\dot H^{1/2}$, there exists a unique variational harmonic extension
$v\in\dot{\mathcal V}(\Pi)$ with trace $\psi$, understood modulo constants, satisfying
$\operatorname{div}(A\nabla v)=0$. If $\psi\in H^1$, then
\begin{equation}\label{B9}
\|\nabla v\|_{\mathcal Y^{1/2}}
\le C_M\|\psi\|_{H^1}.
\end{equation}
\end{lemma}

\begin{proof}
\emph{Trace and zero-trace density.} For $w\in\mathscr C(\overline\Pi)$, Fourier transformation in $x$ and integration in $z$ give
\[
|\xi|\,|\widehat w(\xi,0)|^2
\le
\int_{-\infty}^0
\left(
|\partial_z\widehat w(\xi,z)|^2
+\xi^2|\widehat w(\xi,z)|^2
\right)\,\mathrm dz.
\]
Integrating in $\xi$ yields
\[
\|\operatorname{Tr}w\|_{\dot H^{1/2}}^2
\le \|\nabla w\|_{L^2(\Pi)}^2.
\]
Consequently, the trace extends continuously to $\dot{\mathcal V}(\Pi)$.
To identify this trace with the local boundary trace, fix a bounded half-rectangle $K_0$ adjacent to $z=0$. For an energy Cauchy sequence $(w_n)$, set $c_n=|K_0|^{-1}\int_{K_0}w_n$. On every bounded half rectangle $K$ containing $K_0$, the Poincar\'e inequality with the mean fixed on $K_0$ gives
\[
\|(w_n-c_n)-(w_m-c_m)\|_{H^1(K)}
\le C_K\|\nabla(w_n-w_m)\|_{L^2(K)}.
\]
Thus $w_n-c_n$ has compatible local $H^1$ limits, and the local trace theorem gives compatible boundary traces. Their pairings with compactly supported smooth test functions of zero integral agree with the homogeneous trace: such test functions belong to $\dot H^{-1/2}$, and constants disappear from the pairing. Consequently, the two traces differ only by one additive constant. If the homogeneous trace vanishes, the local boundary trace is a constant; subtracting that constant gives the unique representative with zero local boundary trace. For $\psi\in\dot H^{1/2}$, choose $\psi_n$ with
$\widehat\psi_n\in C_c^\infty(\R\setminus\{0\})$ such that $\psi_n\to\psi$ in $\dot H^{1/2}$. The Poisson extensions $P\psi_n=e^{z\Lambda}\psi_n$ belong to
$H^1(\Pi)$ and hence to $\dot{\mathcal V}(\Pi)$. The Fourier identity below shows that they are Cauchy in the energy norm; their limit defines $P\psi$, with
$\operatorname{Tr}(P\psi)=\psi$. Moreover,
\begin{equation}\label{B10}
\begin{aligned}
\|\nabla P\psi\|_{L^2(\Pi)}^2
&=\|\psi\|_{\dot H^{1/2}}^2,
&&\psi\in\dot H^{1/2},\\
\|\nabla P\psi\|_{\mathcal Y^{1/2}}^2
&=\int_\R|\xi|\langle\xi\rangle
|\widehat\psi(\xi)|^2\,\mathrm d\xi
\le\|\psi\|_{H^1}^2,
&&\psi\in H^1.
\end{aligned}
\end{equation}
Thus $P$ is a bounded right inverse of the homogeneous trace map, satisfying
$\operatorname{Tr}(P\psi)=\psi$ in $\dot H^{1/2}$, with potentials understood modulo constants.

Let $w\in\mathcal V_0$ and choose its representative with zero local boundary trace. For almost every $x$, the function $W(t)=w(x,-t)$ is locally absolutely continuous on $[0,\infty)$, satisfies $W(0)=0$, and has derivative in $L^2(0,\infty)$.
By the Cauchy-Schwarz inequality,
\[
\frac{|W(t)|^2}{t}
\le\int_0^t|W'(r)|^2\,\mathrm dr\longrightarrow0
\quad\text{as }t\downarrow0.
\]
For $0<\delta<R$, set
\[
I_{\delta,R}
=\int_\delta^R\frac{|W(t)|^2}{t^2}\,\mathrm dt.
\]
Integration by parts and the Cauchy-Schwarz inequality give
\begin{align*}
I_{\delta,R}
&=\frac{|W(\delta)|^2}{\delta}
  -\frac{|W(R)|^2}{R}
  +2\operatorname{Re}\int_\delta^R
       \frac{W'(t)\overline{W(t)}}{t}\,\mathrm dt\\
&\le \frac{|W(\delta)|^2}{\delta}
  +2\|W'\|_{L^2(0,R)}I_{\delta,R}^{1/2}.
\end{align*}
This inequality bounds $I_{\delta,R}$ uniformly as $\delta\downarrow0$. Since $|W(\delta)|^2/\delta\to0$ as $\delta\downarrow0$, the monotone convergence theorem yields
\[
\int_0^R\frac{|W(t)|^2}{t^2}\,
\mathrm dt
\le 4\int_0^R|W'(t)|^2
\mathrm dt.
\]
Letting $R\to\infty$ and integrating in $x$, we obtain
\[
\int_\Pi\frac{|w(x,z)|^2}{z^2}\,\mathrm dx\,\mathrm dz
\le4\|\partial_z w\|_{L^2(\Pi)}^2.
\]
Choose $\chi\in C^\infty([0,\infty))$ with
$0\le\chi\le1$, $\chi=0$ on $[0,1]$, and $\chi=1$ on $[2,\infty)$.
Set
\[
w_\varepsilon(x,z)=\chi(-z/\varepsilon)w(x,z).
\]
Since the derivative of the cutoff is supported where
$\varepsilon<-z<2\varepsilon$, we have
\[
\|\nabla(w_\varepsilon-w)\|_{L^2(\Pi)}^2
\le C\int_{\R\times(-2\varepsilon,0)}
\left(
|\nabla w|^2+\frac{|w|^2}{z^2}
\right)\,\mathrm dx\,\mathrm dz
\longrightarrow0
\]
as $\varepsilon\downarrow0$.

Next, choose $\rho\in C_c^\infty(\R^2)$ with $0\le\rho\le1$, $\rho=1$ on $B_1$, and
$\operatorname{supp}\rho\subset B_2$, where $B_R$ denotes the ball of radius $R$ centered at the origin. Set $\rho_R(x,z)=\rho(x/R,z/R)$. For fixed $\varepsilon>0$, the bound $|z|\le2R$ on the support of $\nabla\rho_R$ implies
\[
\begin{aligned}
\|\nabla(\rho_Rw_\varepsilon-w_\varepsilon)\|_{L^2(\Pi)}^2
\le C\int_{\Pi\setminus B_R}
\left(
|\nabla w_\varepsilon|^2
+\frac{|w_\varepsilon|^2}{z^2}
\right)\,\mathrm dx\,\mathrm dz
\longrightarrow0
\end{aligned}
\]
as $R\to\infty$. The function $\rho_Rw_\varepsilon$ belongs to $H^1(\Pi)$, has compact support, and vanishes in a neighborhood of the boundary. Extending it by zero and mollifying at a sufficiently small scale therefore gives approximations in $C_c^\infty(\Pi)$. A diagonal choice of $\varepsilon$, $R$, and the mollification scale proves that $C_c^\infty(\Pi)$ is dense in $\mathcal V_0$ in the energy norm.

Finally, a bounded tangential Fourier multiplier $T$ with symbol $m$ satisfies
\[
\|\nabla Tw\|_{L^2(\Pi)}
\le\|m\|_{L^\infty}\|\nabla w\|_{L^2(\Pi)}.
\]
Since $T$ is a tangential Fourier multiplier, it commutes with spatial differentiation and boundary restriction on smooth functions. By density, the boundedness of $T$ in the energy norm and on $\dot H^{1/2}(\R)$, together with continuity of the trace map, yields
\[
\nabla(Tw)=T\nabla w
\quad\text{in }L^2(\Pi;\R^2),
\qquad
\operatorname{Tr}(Tw)=T\operatorname{Tr}w
\quad\text{in }\dot H^{1/2}(\R),
\]
for every $w\in\dot{\mathcal V}(\Pi)$. Here $T$ acts componentwise on vector fields, and the trace identity is understood modulo additive constants. In particular, $T$ preserves $\mathcal V_0$.

\emph{Variational construction.}
The matrix $A$ is uniformly elliptic:
\[
\lambda_M I\le A\le\lambda_M^{-1}I,
\qquad
\|A-I\|_{H^1}\le C_M,
\qquad \lambda_M>0.
\]
Since $A$ is symmetric, the identities $\det A=1$ and
$\operatorname{tr}A=2+s^2$ imply
\[
(2+\|s\|_{L^\infty}^2)^{-1}|\xi|^2
\le \xi^{\mathsf T}A(x)\xi
\le (2+\|s\|_{L^\infty}^2)|\xi|^2,
\]
for every $\xi\in\R^2$ and almost every $x\in\R$. The embedding $H^1(\R)\hookrightarrow L^\infty(\R)$ and the bound $\|s\|_{H^1}\le M$ therefore give uniform ellipticity with a constant $\lambda_M>0$. Moreover, it follows from the algebra property of $H^1(\R)$
\[
\|A-I\|_{H^1}
\le C\bigl(\|s\|_{H^1}+\|s\|_{H^1}^2\bigr)
\le C_M.
\]
Equipped with the norm $\|\nabla\cdot\|_{L^2(\Pi)}$, the space $\mathcal V_0$ is a Hilbert space on which the bilinear form in \eqref{B7} is bounded and coercive. Since its right-hand side defines a bounded linear functional on $\mathcal V_0$, the Lax-Milgram theorem provides a unique $w\in\mathcal V_0$ satisfying \eqref{B7}. Taking $\theta=w$ we have
\[
\lambda_M\|\nabla w\|_{L^2(\Pi)}^2
\le \int_\Pi A\nabla w\cdot\nabla w
=\int_\Pi F\cdot\nabla w
\le \|F\|_{L^2(\Pi)}\|\nabla w\|_{L^2(\Pi)},
\]
and hence
\[
\|\nabla w\|_{L^2(\Pi)}
\le\lambda_M^{-1}\|F\|_{L^2(\Pi)}.
\]
For $\psi\in\dot H^{1/2}$, the Poisson extension satisfies $\nabla P\psi\in L^2(\Pi)$.
Thus $F=-A\nabla P\psi$ belongs to $L^2(\Pi;\R^2)$, and the preceding construction applies. Setting $v=P\psi+w$, we obtain $\operatorname{Tr}v=\psi$ and
\[
\int_\Pi A\nabla v\cdot\nabla\theta=0,
\qquad \theta\in\mathcal V_0.
\]
In particular, $\operatorname{div}(A\nabla v)=0$ in the sense of distributions.
If $v_1$ and $v_2$ are two variational solutions in $\dot{\mathcal V}(\Pi)$ with trace $\psi$, their difference belongs to $\mathcal V_0$. Testing the difference equation with $v_1-v_2$ and using coercivity yields $\nabla(v_1-v_2)=0$, proving uniqueness modulo additive constants.

\emph{Tangential regularity.}
Let $T=T_{s,N}$ be the truncated multiplier from the preceding commutator lemma. Since $T$ is self-adjoint and preserves $\mathcal V_0$, we may test \eqref{B7} with $T^2w$. This yields
\[
\int_\Pi A\nabla Tw\cdot\nabla Tw
=
\int_\Pi TF\cdot\nabla Tw
-\int_\Pi[T,A]\nabla w\cdot\nabla Tw,
\]
and thus
\begin{equation}\label{B11}
\lambda_M\|\nabla Tw\|_{L^2(\Pi)}
\le
\|TF\|_{L^2(\Pi)}
+\|[T,A]\nabla w\|_{L^2(\Pi)}.
\end{equation}

First take $s=3/8$: integrating the uniform subcritical commutator estimate \eqref{B1} in $z$ bounds the last term by $C_M\|\nabla w\|_{L^2(\Pi)}$, and Fatou's lemma together with the energy estimate then implies, as $N\to\infty$,
\[
\|\nabla w\|_{\mathcal Y^{3/8}}
\le C_M\|F\|_{\mathcal Y^{3/8}}.
\]
Next set $s=1/2$ in \eqref{B11}. By \eqref{B2}, the commutator term is bounded by
$C_M\|\nabla w\|_{\mathcal Y^{3/8}}$, which is controlled by the preceding estimate.
Letting $N\to\infty$ and applying Fatou's lemma then proves \eqref{B8}.

Finally, for $\psi\in H^1$, it follows from \eqref{B3} and \eqref{B10} that
\[
\|A\nabla P\psi\|_{\mathcal Y^{1/2}}
\le C_M\|\psi\|_{H^1}.
\]
Applying \eqref{B8} to the zero-trace correction $w$ and incorporating the Poisson extension therefore proves \eqref{B9}.
\end{proof}

\begin{proposition}\label{prop:B-DN}
The variational definition
\begin{equation}\label{B12}
\langle G(\eta)\psi,\varphi\rangle
=\int_\Pi A_\eta\nabla v_\psi\cdot\nabla v_\varphi
\end{equation}
defines a bounded map $\dot H^{1/2}\to\dot H^{-1/2}$. For $\psi\in H^1$, it is the strong $L^2$ trace $(A_\eta\nabla v_\psi)_z|_{z=0}$, with
\begin{equation}\label{B13}
\|G(\eta)\psi\|_{L^2}\le C_M\|\psi\|_{H^1}.
\end{equation}
If $\|\eta_j\|_{H^2}\le M$, then
\begin{equation}\label{B14}
\|(G(\eta_1)-G(\eta_2))\psi\|_{L^2}
\le C_M\|\eta_1-\eta_2\|_{H^2}\|\psi\|_{H^1}.
\end{equation}
The Dirichlet form is comparable to $\|\psi\|_{\dot H^{1/2}}^2$.
\end{proposition}
\begin{proof}
The variational construction and the Cauchy-Schwarz inequality establish boundedness of the form \eqref{B12}, while uniform ellipticity and the minimizing property of $v_\psi$, with the flat Poisson extension as a competitor, yield
\[
\lambda_M\|\psi\|_{\dot H^{1/2}}^2
\le \int_\Pi A_\eta\nabla v_\psi\cdot\nabla v_\psi
\le \lambda_M^{-1}\|\psi\|_{\dot H^{1/2}}^2.
\]
For $\psi\in H^1$, Lemma~\ref{lem:B-shift} gives $\nabla v_\psi\in\mathcal Y^{1/2}$.
Since the entries of $A_\eta-I$ belong to $H^1$, the product estimate \eqref{B3} gives $Q=A_\eta\nabla v_\psi\in\mathcal Y^{1/2}$; as the variational equation gives $\operatorname{div}Q=0$, Lemma~\ref{lem:B-flux} supplies the strong $L^2$ normal trace $Q_z(0)$ satisfying \eqref{B13}.

For $\phi\in C_c^\infty(\R)$, choose a smooth extension $\Phi$ with compact support in $\overline\Pi$ and trace $\phi$. Since $\operatorname{div}Q=0$, integration by parts on $\{z<-\varepsilon\}$, followed by $\varepsilon\downarrow0$, gives
\[
\int_\Pi Q\cdot\nabla\Phi
=\int_\R Q_z(x,0)\phi(x)\,\mathrm dx.
\]
The boundary limit follows from the strong $L^2$ trace of $Q_z$. Moreover, $\Phi-v_\phi\in\mathcal V_0$, so the variational equation for $v_\psi$ yields
\[
\int_\Pi Q\cdot\nabla\Phi
=\int_\Pi Q\cdot\nabla v_\phi
=\langle G(\eta)\psi,\phi\rangle.
\]
Thus $Q_z(0)$ represents $G(\eta)\psi$ in $L^2$. By density, this identity holds for every $\phi\in H^1$.

To prove continuity with respect to the graph, let $A_j=A_{\eta_j}$, where $\|\eta_j\|_{H^2}\le M$, and let $v_j$ be the variational harmonic extensions of the same trace $\psi\in H^1$. Then $d=v_1-v_2$ has zero trace and satisfies
\[
\operatorname{div}(A_1\nabla d)
=-\operatorname{div}\bigl((A_1-A_2)\nabla v_2\bigr).
\]
By \eqref{B6} and the algebra property of $H^1$,
\[
\|A_1-A_2\|_{H^1}
\le C_M\|\eta_1-\eta_2\|_{H^2}.
\]
It follows from Lemma~\ref{lem:B-comm} and Lemma~\ref{lem:B-shift} that
\[
\|\nabla d\|_{\mathcal Y^{1/2}}
\le C_M\|\eta_1-\eta_2\|_{H^2}\|\psi\|_{H^1}.
\]
Writing $A_1=I+(A_1-I)$ and applying the same product estimate gives
\[
\|Q_1-Q_2\|_{\mathcal Y^{1/2}}
\le C_M\|\eta_1-\eta_2\|_{H^2}\|\psi\|_{H^1},
\qquad
Q_1-Q_2=A_1\nabla d+(A_1-A_2)\nabla v_2.
\]
Because this flux difference is divergence-free with normal trace $(G(\eta_1)-G(\eta_2))\psi$, the trace estimate \eqref{B4} proves \eqref{B14} without a shape-derivative formula.
\end{proof}

\begin{proposition}\label{prop:B-Rellich}
Let $\eta\in H^2$ with $\|\eta\|_{H^2}\le M$. For every $\psi\in H^1$,
\begin{equation}\label{B15}
C_M^{-1}\|\psi_x\|_{L^2}
\le \|G(\eta)\psi\|_{L^2}
\le C_M\|\psi_x\|_{L^2},
\end{equation}
where the constants depend only on $\|\eta_x\|_\infty$. With $a=\psi_x$, $h=G(\eta)\psi$, and $s=\eta_x$, the following identity holds:
\begin{equation}\label{B16}
\int_\R\frac{h^2-a^2+2sah}{1+s^2}\dd=0.
\end{equation}
\end{proposition}
\begin{proof}
First assume that $\eta,\psi\in\mathcal S(\R)$ are real valued.
Write $A=A_\eta$ and let
\[
v=P\psi+w
\]
be the flattened variational harmonic extension, where
$w\in\mathcal V_0$ solves \eqref{B7} with
$F=-A\nabla P\psi$.

We first establish the regularity needed for the boundary
calculation. For $\delta\ne0$, define
\[
D_\delta g(x,z)
=\frac{g(x+\delta,z)-g(x,z)}{\delta}.
\]
Tangential translations preserve $\mathcal V_0$.
Subtracting the translated variational equations gives
\[
\int_\Pi A\nabla D_\delta w\cdot\nabla\theta
=
\int_\Pi
\left(
D_\delta F
-(D_\delta A)\nabla w(x+\delta,z)
\right)\cdot\nabla\theta,
\qquad \theta\in\mathcal V_0.
\]
Taking $\theta=D_\delta w$ and using ellipticity, we have
\[
\|\nabla D_\delta w\|_{L^2(\Pi)}
\le C_\eta\left(
\|\partial_xF\|_{L^2(\Pi)}
+\|\partial_xA\|_{L^\infty}
 \|\nabla w\|_{L^2(\Pi)}
\right).
\]
Here $C_\eta$ is finite for the fixed smooth graph.
The energy estimate and the Poisson extension identities imply
\[
\|\nabla w\|_{L^2(\Pi)}
+\|\partial_xF\|_{L^2(\Pi)}
+\|\nabla\partial_xP\psi\|_{L^2(\Pi)}
\le C_\eta\|\psi\|_{H^{3/2}}.
\]
Passing to the limit in the difference quotients therefore yields
\[
\|\partial_x\nabla v\|_{L^2(\Pi)}
\le C_\eta\|\psi\|_{H^{3/2}}.
\]
Since $\operatorname{div}(A\nabla v)=0$, we also have
\[
(1+s^2)v_{zz}
=2s\,v_{xz}-v_{xx}+s_xv_z,
\qquad s=\eta_x.
\]
By the boundedness of $s$ and $s_x$, it follows
\[
\|\nabla v\|_{H^1(\Pi)}
\le C_\eta\|\psi\|_{H^{3/2}}.
\]
For the Schwartz data considered here, the tangential difference-quotient estimate can be iterated: derivatives of the coefficients are bounded, and the displayed equation for $v_{zz}$ recovers successive normal derivatives. This yields smoothness of $v$ up to $z=0$ locally. For the Rellich calculation, the established bound $\nabla v\in H^1(\Pi)$ suffices. Indeed, with $\Phi_\eta(x,z)=(x,z+\eta(x))$ and $s=\eta_x$,
\[
(\nabla\Psi)\circ\Phi_\eta=(v_x-sv_z,v_z).
\]
Because $s\in W^{1,\infty}$ and $\Phi_\eta$ is bi-Lipschitz,
\[
\|\nabla\Psi\|_{H^1(\Omega_\eta)}
\le C(\|s\|_{W^{1,\infty}})\|\nabla v\|_{H^1(\Pi)}<\infty.
\]
The trace theorem therefore gives
$V=\Psi_x|_{y=\eta}$ and $\beta=\Psi_y|_{y=\eta}$ in $L^2(\mathbb R)$.

Define the Rellich vector field
\[
\mathscr R
=2\Psi_y\nabla\Psi-|\nabla\Psi|^2e_y.
\]
Since $\nabla\Psi\in H^1(\Omega_\eta)$, we have $\mathscr R\in W^{1,1}(\Omega_\eta;\R^2)$.
By harmonicity, it holds
\[
\operatorname{div}\mathscr R
=2\Psi_y\Delta\Psi=0.
\]
Choose $\vartheta\in C_c^\infty(\R^2)$ with $0\le\vartheta\le1$ and $\vartheta=1$ on $B_1$, and set $\vartheta_R(x,y)=\vartheta(x/R,y/R)$. Applying the divergence theorem to $\vartheta_R\mathscr R$, we obtain
\[
\int_\R
\vartheta_R(x,\eta(x))
\left(\beta^2-V^2-2s\beta V\right)\,\mathrm dx
=
\int_{\Omega_\eta}
\nabla\vartheta_R\cdot\mathscr R\,\mathrm dx\,\mathrm dy.
\]
Indeed, the boundary measure and outward normal satisfy
$n_\eta\,\mathrm dS=(-s,1)\,\mathrm dx$.
The right-hand side satisfies
\[
\left|
\int_{\Omega_\eta}
\nabla\vartheta_R\cdot\mathscr R\,\mathrm dx\,\mathrm dy
\right|
\le
\frac{C}{R}\|\nabla\Psi\|_{L^2(\Omega_\eta)}^2
\longrightarrow0.
\]
Moreover, $\beta^2-V^2-2s\beta V\in L^1(\R)$, due to $V,\beta\in L^2$ and $s\in L^\infty$.
It follows by the dominated convergence on the boundary
\[
\int_\R
\left(\beta^2-V^2-2s\beta V\right)\,\mathrm dx=0.
\]
Using
\[
a=V+s\beta,\qquad h=\beta-sV,
\]
we obtain
\[
\beta^2-V^2-2s\beta V
=
\frac{h^2-a^2+2sah}{1+s^2},
\]
which proves \eqref{B16} for smooth data.

For general real valued $\eta\in H^2$ and $\psi\in H^1$,
choose real valued Schwartz approximations
$\eta_n\to\eta$ in $H^2$ and $\psi_n\to\psi$ in $H^1$.
Set
\[
s_n=\partial_x\eta_n,\qquad
a_n=\partial_x\psi_n,\qquad
h_n=G(\eta_n)\psi_n.
\]
The graph norms are uniformly bounded, and
\eqref{B13}--\eqref{B14} give
\[
\|h_n-h\|_2
\le
C_M\|\psi_n-\psi\|_{H^1}
+C_M\|\eta_n-\eta\|_{H^2}\|\psi\|_{H^1}
\longrightarrow0.
\]
Also, $a_n\to a$ in $L^2$ and $s_n\to s$ in $L^\infty$.
It follows that
\[
\frac{h_n^2-a_n^2+2s_na_nh_n}{1+s_n^2}
\longrightarrow
\frac{h^2-a^2+2sah}{1+s^2}
\quad\text{in }L^1(\R).
\]
Strong convergence of the smooth approximations in $H^2\times H^1$ permits passage to the limit in the smooth identities, proving \eqref{B16} at the stated regularity. The estimates used here depend only on uniform $H^2$ bounds on the graphs and $H^1$ bounds on the traces.

Let
\[
q=\|s\|_\infty,\qquad
X^2=\int_\R\frac{h^2}{1+s^2}\,\mathrm dx,\qquad
Y^2=\int_\R\frac{a^2}{1+s^2}\,\mathrm dx.
\]
The identity and the Cauchy-Schwarz inequality yield
$|X^2-Y^2|\le 2qXY$, hence
\[
X\le\bigl(q+\sqrt{1+q^2}\bigr)Y,\qquad
Y\le\bigl(q+\sqrt{1+q^2}\bigr)X.
\]
Since
\[
(1+q^2)^{-1}\le(1+s^2)^{-1}\le1,
\]
the weighted norms are equivalent to the standard $L^2$ norms, with constants depending only on $q$, which proves \eqref{B15}.
\end{proof}

\begin{proposition}\label{prop:B-resolvent}
For every real valued $\eta\in H^2$, $\Id+G(\eta):H^1\to L^2$ is a bounded isomorphism. Its inverse $R_\eta=(\Id+G(\eta))^{-1}$ satisfies
\begin{equation}\label{B17}
\|R_\eta g\|_{H^1}\le C_M\|g\|_{L^2},\qquad
\|R_\eta g\|_{L^2}\le\|g\|_{L^2},
\end{equation}
for all $g\in L^2$, provided $\|\eta\|_{H^2}\le M$. As an operator on $L^2$, $G(\eta)$ is nonnegative and self-adjoint, with operator domain $H^1$ and form domain $H^{1/2}$. Moreover, for $\|\eta_j\|_{H^2}\le M$,
\begin{equation}\label{B18}
\|(R_{\eta_1}-R_{\eta_2})g\|_{H^1}
\le C_M\|\eta_1-\eta_2\|_{H^2}\|g\|_{L^2}.
\end{equation}
\end{proposition}
\begin{proof}
By nonnegativity of the Dirichlet form, for $g=\varphi+G(\eta)\varphi$ with $\varphi\in H^1$,
\[
\|\varphi\|_{L^2}^2\le\langle g,\varphi\rangle,
\qquad
\|\varphi\|_{L^2}\le\|g\|_{L^2}.
\]
Combining the lower Rellich estimate with $G(\eta)\varphi=g-\varphi$ gives $\|\varphi_x\|_2\le 2C_M\|g\|_2$ and hence
\[
\|\varphi\|_{H^1}
\le C_M\|(\Id+G(\eta))\varphi\|_{L^2},
\]
uniformly for $\|\eta\|_{H^2}\le M$, independently of surjectivity.

Set $T_\theta=\Id+G(\theta\eta)$ for $0\le\theta\le1$. Proposition~\ref{prop:B-DN} gives continuity of $\theta\mapsto T_\theta$ in $\mathcal L(H^1,L^2)$, and the preceding estimate holds uniformly in $\theta$. Fourier transformation shows that $T_0=\Id+|D|$ is an isomorphism, while the Neumann series shows that the set of invertible parameters is open. For closedness, let $\theta_n\to\theta_*$ with each $T_{\theta_n}$ invertible and, for $g\in L^2$, set $\varphi_n=T_{\theta_n}^{-1}g$. The uniform estimate implies
\[
\|\varphi_n-\varphi_m\|_{H^1}
\le C_M\|(T_{\theta_m}-T_{\theta_n})\varphi_m\|_2
\le C_M\|T_{\theta_m}-T_{\theta_n}\|_{H^1\to L^2}\|g\|_2.
\]
Thus $\varphi_n$ converges in $H^1$ to a function $\varphi$ satisfying $T_{\theta_*}\varphi=g$; the same a priori estimate gives injectivity, so $T_{\theta_*}$ is invertible. Since $[0,1]$ is connected, $T_1$ is invertible and satisfies \eqref{B17}.

Symmetry and nonnegativity of $G=G(\eta)$ follow from the Dirichlet form. To prove self-adjointness, let $h\in D(G^*)$, set $f=h+G^*h$, and take $w=R_\eta f\in H^1$. For every $\varphi\in H^1$, using the definition of the adjoint and symmetry, it follows
\[
\langle h,(\Id+G)\varphi\rangle
=\langle f,\varphi\rangle
=\langle w,(\Id+G)\varphi\rangle.
\]
As $\Id+G:H^1\to L^2$ is surjective, $h=w\in H^1$ and $G^*h=Gh$; combined with symmetry, this gives $D(G^*)=D(G)=H^1$. Proposition~\ref{prop:B-DN} further gives
\[
\|\varphi\|_2^2+\Gform_\eta(\varphi,\varphi)
\asymp_M \|\varphi\|_{H^{1/2}}^2.
\]
Because $H^1$ is dense in $H^{1/2}$, the closed quadratic form has domain $H^{1/2}$, which is therefore the form domain of $G$. Finally, with $G_j=G(\eta_j)$ and $R_j=R_{\eta_j}$, the resolvent identity
\[
R_1-R_2=R_1(G_2-G_1)R_2,
\]
combined with \eqref{B14} and \eqref{B17} yields \eqref{B18}.
\end{proof}

\begin{theorem}[Shape derivative at the required regularity]
\label{thm:B-shape}
The map $\eta\mapsto G(\eta)$ is continuously Fr\'echet differentiable from $H^2$ into $\mathcal L(H^1,L^2)$. Its derivative satisfies
\begin{equation}\label{B19}
\|DG(\eta)[k]\psi\|_2
\le C_M\|k\|_{H^2}\|\psi\|_{H^1}.
\end{equation}
The derivative is given distributionally by
\begin{equation}\label{B20}
DG(\eta)[k]\psi
=-G(\eta)(k\beta)-\partial_x(k\nu),\qquad
\beta=\frac{G(\eta)\psi+\eta_x\psi_x}{1+\eta_x^2},
\quad
\nu=\psi_x-\eta_x\beta.
\end{equation}
In the first term, $G(\eta)$ denotes its bounded extension from $L^2$ to $H^{-1}$ by duality; the two terms separately belong to $H^{-1}$, whereas their sum belongs to $L^2$ by \eqref{B19}.
\end{theorem}
\begin{proof}
\emph{Material derivative and its flux.}
Fix $\eta,k\in H^2$ and $\psi\in H^1$, with $\|\eta\|_{H^2}\le M$. Set $A=A_\eta$ and
$A_\varepsilon=A_{\eta+\varepsilon k}$, and let $v,v_\varepsilon$ be the corresponding flattened harmonic extensions of the same trace $\psi$. Direct expansion gives
\[
A_\varepsilon=A+\varepsilon\dot A+\varepsilon^2A_2,\qquad
\dot A=
\begin{pmatrix}
0&-k_x\\
-k_x&2\eta_xk_x
\end{pmatrix},
\qquad
A_2=
\begin{pmatrix}
0&0\\
0&k_x^2
\end{pmatrix}.
\]
By the algebra property of $H^1$, we obtain
\[
\|\dot A\|_{H^1}\le C_M\|k\|_{H^2},
\qquad
\|A_2\|_{H^1}\le C\|k\|_{H^2}^2.
\]
Let $\dot v\in\mathcal V_0$ be the unique variational solution of
\[
\operatorname{div}(A\nabla\dot v)
=-\operatorname{div}(\dot A\nabla v).
\]
The product and zero-trace shift estimates imply
\begin{equation}\label{B21}
\|\nabla\dot v\|_{\mathcal Y^{1/2}}
\le C_M\|k\|_{H^2}\|\psi\|_{H^1}.
\end{equation}
Consequently, the material flux variation
\[
\dot Q=A\nabla\dot v+\dot A\nabla v
\]
is divergence-free and satisfies the same $\mathcal Y^{1/2}$ estimate, so the normal trace bound gives
\[
\|\dot Q_z(0)\|_2
\le C_M\|k\|_{H^2}\|\psi\|_{H^1}.
\]

\emph{Quadratic remainder.} Assume $|\varepsilon|\|k\|_{H^2}\le1$, so that the perturbed graphs remain uniformly bounded in $H^2$. The interior difference estimate in the proof of Proposition~\ref{prop:B-DN} applied to $\eta+\varepsilon k$ and $\eta$, gives
\[
\|\nabla(v_\varepsilon-v)\|_{\mathcal Y^{1/2}}
\le C_M|\varepsilon|\|k\|_{H^2}\|\psi\|_{H^1}.
\]
The zero-trace remainder $r_\varepsilon=v_\varepsilon-v-\varepsilon\dot v$
satisfies
\[
\operatorname{div}(A\nabla r_\varepsilon)
=-\operatorname{div}\bigl(
\varepsilon\dot A\nabla(v_\varepsilon-v)
+\varepsilon^2A_2\nabla v_\varepsilon\bigr).
\]
The product and zero-trace shift estimates therefore yield
\[
\|\nabla r_\varepsilon\|_{\mathcal Y^{1/2}}
\le C_M\varepsilon^2\|k\|_{H^2}^2\|\psi\|_{H^1}.
\]
The exact flux remainder is
\[
Q_\varepsilon-Q-\varepsilon\dot Q
=A\nabla r_\varepsilon
+\varepsilon\dot A\nabla(v_\varepsilon-v)
+\varepsilon^2A_2\nabla v_\varepsilon.
\]
This remainder is divergence-free and satisfies the same bound in $\mathcal Y^{1/2}$. Using the normal trace estimate, it follows
\begin{equation}\label{B22}
\|G(\eta+\varepsilon k)\psi-G(\eta)\psi
-\varepsilon\dot Q_z(0)\|_2
\le C_M\varepsilon^2\|k\|_{H^2}^2\|\psi\|_{H^1}.
\end{equation}
Because $\dot Q_z(0)$ depends linearly on $k$ and satisfies the preceding material flux bound, setting $\varepsilon=1$ and letting $k\to0$ proves Fr\'echet differentiability in $\mathcal L(H^1,L^2)$, with $DG(\eta)[k]\psi=\dot Q_z(0)$.

To establish continuity of the derivative, compare two base graphs with $\|\eta_j\|_{H^2}\le M$, keeping $k$ and $\psi$ fixed. Set $A_j=A_{\eta_j}$ and
\[
\dot A_j
=\left.\frac{\mathrm d}{\mathrm d\varepsilon}
A_{\eta_j+\varepsilon k}\right|_{\varepsilon=0}.
\]
Subtracting the variational equations for $\dot v_1$ and $\dot v_2$ and regrouping the terms yields
\begin{align*}
\operatorname{div}\bigl(A_1\nabla(\dot v_1-\dot v_2)\bigr)
&={}-\operatorname{div}\bigl(
(A_1-A_2)\nabla\dot v_2\\
&+(\dot A_1-\dot A_2)\nabla v_2
+\dot A_1\nabla(v_1-v_2)\bigr).
\end{align*}
Using
\[
\|\dot A_1-\dot A_2\|_{H^1}
\le C\|\eta_1-\eta_2\|_{H^2}\|k\|_{H^2},
\]
together with the preceding coefficient and solution estimates, the product and shift estimates give
\[
\|\nabla(\dot v_1-\dot v_2)\|_{\mathcal Y^{1/2}}
\le C_M\|\eta_1-\eta_2\|_{H^2}
\|k\|_{H^2}\|\psi\|_{H^1}.
\]
The divergence free flux difference
\begin{align*}
\dot Q_1-\dot Q_2
={}&A_1\nabla(\dot v_1-\dot v_2)
+(A_1-A_2)\nabla\dot v_2\\
&+(\dot A_1-\dot A_2)\nabla v_2
+\dot A_1\nabla(v_1-v_2),
\end{align*}
therefore satisfies
\[
\|\dot Q_1-\dot Q_2\|_{\mathcal Y^{1/2}}
\le C_M\|\eta_1-\eta_2\|_{H^2}
\|k\|_{H^2}\|\psi\|_{H^1}.
\]
Taking its normal trace yields
\[
\|(DG(\eta_1)[k]-DG(\eta_2)[k])\psi\|_2
\le C_M\|\eta_1-\eta_2\|_{H^2}
\|k\|_{H^2}\|\psi\|_{H^1}.
\]
Hence the derivative is locally Lipschitz with respect to the graph in the stated operator norm.

\emph{Identification of the boundary formula.} For smooth data, let $\Psi$ be the physical harmonic extension and $\delta\Psi$ its Eulerian variation at fixed physical coordinates. Differentiating the harmonic equation and the fixed trace condition, it follows
\[
\Delta\delta\Psi=0,\qquad
\delta\Psi|_{y=\eta}=-k\beta.
\]
Equivalently,
\[
\delta\Psi(x,z+\eta(x))
=\dot v(x,z)-k(x)v_z(x,z).
\]
Let
\[
q=\dot v-kv_z
=\delta\Psi\circ\Phi_\eta,
\qquad
\Phi_\eta(x,z)=(x,z+\eta(x)).
\]
Since $\det D\Phi_\eta=1$, the change of variables gives
\[
\int_{\Omega_\eta}|\nabla_{x,y}\delta\Psi|^2
\,\mathrm dx\,\mathrm dy
=
\int_\Pi A_\eta\nabla q\cdot\nabla q
\,\mathrm dx\,\mathrm dz.
\]
The upper ellipticity bound for $A_\eta$ therefore yields
\[
\|\nabla_{x,y}\delta\Psi\|_{L^2(\Omega_\eta)}
\le C_\eta\left(
\|\nabla\dot v\|_{L^2(\Pi)}
+\|\nabla(kv_z)\|_{L^2(\Pi)}
\right).
\]
The first term is finite by the variational estimate
for $\dot v$. For the smooth data considered here,
the previously established regularity of $v$ gives
$v_z\in H^1(\Pi)$, and
\[
\|\nabla(kv_z)\|_{L^2(\Pi)}
\le
\|k_x\|_{L^\infty}\|v_z\|_{L^2(\Pi)}
+\|k\|_{L^\infty}\|\nabla v_z\|_{L^2(\Pi)}
<\infty.
\]
Thus $\delta\Psi$ has finite Dirichlet energy.
Together with harmonicity and the trace $-k\beta$, this identifies $\delta\Psi$ as its variational harmonic extension by uniqueness. The additional smooth-data bounds in this identification step are used only before approximation; the uniform $H^2\times H^2\times H^1$ estimates were established above. Differentiating
$G(\eta)\psi=(\Psi_y-\eta_x\Psi_x)|_{y=\eta}$ yields
\begin{align*}
DG(\eta)[k]\psi
&=-G(\eta)(k\beta)
+k(\Psi_{yy}-\eta_x\Psi_{xy})|_{y=\eta}-k_x\nu\\
&=-G(\eta)(k\beta)-k\partial_x\nu-k_x\nu,
\end{align*}
where we have used $\Psi_{yy}=-\Psi_{xx}$ and
\[
\partial_x\nu
=(\Psi_{xx}+\eta_x\Psi_{xy})|_{y=\eta}.
\]
This proves \eqref{B20} for smooth data.

For general data, choose smooth triples converging to $(\eta,k,\psi)$ in $H^2\times H^2\times H^1$. The derivative bounds and continuity with respect to the graph imply strong $L^2$ convergence of the left-hand side. Proposition~\ref{prop:B-DN} and convergence of the graph slopes in $L^\infty$ give convergence of $\beta_n$ and
$\nu_n$ in $L^2$. Since $k_n\to k$ in $L^\infty$, the products $k_n\beta_n$ and $k_n\nu_n$ also converge in $L^2$. By symmetry, \eqref{B13} and \eqref{B14} imply boundedness
and continuity with respect to the graph for $G(\eta):L^2\to H^{-1}$ by duality.
Thus both terms on the right-hand side converge in $H^{-1}$, proving \eqref{B20} at the stated regularity, and their sum belongs to $L^2$ and satisfies \eqref{B19}.
\end{proof}

\begin{remark}
The divergence free material flux $\dot Q$ gives \eqref{B19} through the normal trace estimate, which requires only half a tangential derivative in the bulk. By contrast, separate estimates for the two terms in \eqref{B20} give only $H^{-1}$ control; the $L^2$ bound therefore depends on their cancellation, which is explicit at the flat graph:
\[
DG(0)[k]\psi
=-|D|(k|D|\psi)-\partial_x(k\psi_x).
\]
Writing $\xi$ for the output frequency and $\zeta$ for the frequency of $\psi$, so that $\xi-\zeta$ is the frequency of $k$, the bilinear Fourier symbol is
\[
m(\xi,\zeta)=\xi\zeta-|\xi||\zeta|.
\]
This symbol vanishes when $\xi\zeta\ge0$. If $\xi\zeta<0$, then $|\xi-\zeta|=|\xi|+|\zeta|$ and $|m(\xi,\zeta)|=2|\xi||\zeta|$. Consequently,
\[
|m(\xi,\zeta)|
\le 2|\xi-\zeta||\zeta|
\qquad\text{for all }\xi,\zeta\in\R.
\]
It follows from Young's convolution inequality and Plancherel's theorem that
\[
\|DG(0)[k]\psi\|_2
\le C\bigl\||\xi|\widehat k(\xi)\bigr\|_{L^1_\xi}
       \|\psi_x\|_2
\le C\|k\|_{H^2}\|\psi_x\|_2.
\]
For general $H^2$ graphs, the flux argument supplies the corresponding $L^2$ gain in \eqref{B19}, with constants controlled by the graph's $H^2$ bound and no slope-smallness assumption.
\end{remark}

\noindent\textbf{Funding}
The authors declare that no external funding was
received for the research, authorship, or publication
of this article.

\vskip 6mm
\noindent{\bf Declaration of competing interests}
The authors declare no competing interests.
\vskip 6mm

\noindent\textbf{Data availability statement.}
Data sharing is not applicable to this article as no data sets
were generated or analysed during the current study.

\end{document}